\documentclass[a4paper, reqno]{amsart}
\usepackage{amsthm, amsmath, amssymb, amsfonts, amsopn, color,  geometry}
\usepackage{enumerate}
 
  \usepackage{graphicx}
    \usepackage{subfig}
    \usepackage{floatrow}
 
\usepackage[all]{xy}
 
\makeatletter
\newcommand*{\rom}[1]{\expandafter\@slowromancap\romannumeral #1@}
\makeatother

\makeatletter
\@namedef{subjclassname@2020}{%
  \textup{2020} Mathematics Subject Classification}
\makeatother

\newcommand{\e}{\varepsilon}

\newcommand{\Z}{\mathbb{Z}}
\newcommand{\R}{\mathbb{R}}

\newcommand{\N}{\mathbb{N}}

\newcommand{\be}{\begin{equation}}

\newcommand {\ee} {\end{equation}}
\theoremstyle{definition}
\newtheorem*{defn}{Definition}
\theoremstyle{plain}
\newtheorem*{thm*}{Theorem}
\theoremstyle{plain}
\newtheorem{thm}{Theorem}[section]
\theoremstyle{plain}
\newtheorem{lem}[thm]{Lemma}
\theoremstyle{plain}
\newtheorem{cor}[thm]{Corollary}
\theoremstyle{plain}
\newtheorem{prop}[thm]{Proposition}
\theoremstyle{remark}
\newtheorem{rmk}[thm]{Remark}

\begin{document}
\title[Multi-bump standing waves for  nonlinear Schr\"{o}dinger equations]{
Multi-bump standing waves for  nonlinear Schr\"{o}dinger equations with a general nonlinearity: the topological effect of potential wells}
\author{ Sangdon Jin}
\address[Sangdon Jin]  {Stochastic Analysis and Application Research Center, KAIST, 291 Daehak-ro, Yuseong-gu, Taejon, 305-701, Republic of Korea}
\email{sdjin@kaist.ac.kr}

\begin{abstract}
In this article, we are interested in multi-bump solutions of the singularly perturbed problem
\begin{equation*}
-\epsilon^2\Delta v+V(x)v=f(v) \ \  \mbox{ in }\ \  \R^N.
\end{equation*}
Extending  previous results \cite{B, DLY,W1}, we prove the existence of multi-bump solutions for an optimal class of nonlinearities $f$   satisfying the Berestycki-Lions conditions  and, notably,  also for more general classes of  potential wells    than those previously studied. We devise two novel topological arguments to deal with  general classes of  potential wells. Our results prove  the existence of multi-bump  solutions in which  the centers of  bumps converge toward potential wells as $\epsilon\rightarrow 0$. Examples  of potential wells  include the following: the union of two compact smooth submanifolds of $\R^N$ where these two submanifolds meet at the origin and   an embedded topological submanifold of $\R^N$. 
\end{abstract}

\date{\today}
\maketitle

\section {Introduction} 
 One of the best-known nonlinear partial differential equations is the  nonlinear Schr\"{o}dinger equation, which appears   in a variety of physical contexts including Bose-Einsten condensation \cite{PS}, nonlinear atom optics \cite{M1, M2} and many others. 
This paper is devoted to the study  of  the  standing wave solutions  to the nonlinear Schr\"{o}dinger equation
\begin{equation}\label{b1}
i\hbar\frac{\partial\psi}{\partial t}+\frac{\hbar^2}{2}\Delta \psi -V(x)\psi+f(\psi)=0,\ \ (t,x)\in \R\times \R^N,
\end{equation}
where $\hbar$ denotes the Planck constant, $i$ is the imaginary unit and assume that $f(\exp(i\theta)s)=\exp(i\theta)f(s)$ for $s, \theta \in \R$. Considering the standing wave solutions of the form $\psi(x,t)=\exp(-iEt/\hbar)v(x)$, $\psi$ is a standing wave  solution of \eqref{b1} if and only if  the function $v$ satisfies 
\begin{equation*}
  \frac{\hbar^2}{2}\Delta v-(V(x)-E)v+f(v)=0 \ \mbox{ in }\ \R^N.
\end{equation*}
We examine standing waves of the nonlinear Schr\"{o}dinger equation \eqref{b1} for small $\hbar>0$, which are referred to as semi-classical states.
 As a result, we focus on the following equation
\begin{equation}\label{SPP}
  \begin{cases}
  \epsilon^2\Delta v-V(x)v+f(v)=0 & \mbox{in } \R^N, \\
  v(x)>0 & \mbox{in } \R^N,\\
  v(x)\rightarrow0  & \mbox{as }|x|\rightarrow \infty.
\end{cases}
\end{equation}
  Throughout the paper, we assume that the potential $V$ satisfies
\begin{enumerate}[(V1)]
  \item $V\in C(\R^N,\R)$ and $0<V_0=\inf_{x\in \R^N}V(x)\le \sup_{x\in \R^N}V(x)<\infty$;
  \item there is a bounded domain $O$ such that $m=\inf_{x\in O}V(x)<\inf_{x\in \partial O}V(x)$.
\end{enumerate}
Set $\mathcal{M}\equiv \{ x\in O\ |\ V(x)=m\}$.  Studying the solutions to \eqref{SPP} requires considering a limiting problem \begin{equation}\label{b3}
  \Delta u-mu+f(u)=0 \mbox{ in } \R^N.
\end{equation}
  Berestycki and Lions showed in their celebrated paper \cite{BL} that the limiting problem \eqref{b3} has a least-energy solution under the following conditions 
\begin{enumerate}[(f1)]
  \item $f\in C(\R,\R)$ and $f(0)=\lim_{t\rightarrow 0}\frac{f(t)}{t}=0;$
  \item there exist $\bar{C}>0$ and $p\in (1, \frac{N+2}{N-2})$ such that
  \[
  |f(t)|\le \bar{C}(1+t^p) \mbox{ for all } t\ge 0;
  \]
  \item there exists a $t_0>0$ such that
  \[ F(t_0)>\frac{1}{2}m t_0^2, \ \ \mbox{ where } \ \ F(t)=\int_{0}^{t}f(s)ds.
  \]
\end{enumerate}

The groundbreaking work of Floer and Weinstein \cite{FW} is the first research in this direction. Adopting the Lyapunov-Schmidt method, they showed that as $\e\rightarrow0$ a family of single peak solutions of \eqref{SPP} concentrates around a non-degenerate critical point of $V$.  The Lyapunov-Schmidt method requires a linearized non-degeneracy condition, namely, that for a critical point $x_0\in \R^N$ of $V$ and a positive solution $U_{x_0}\in H^1(\R^N)$ of a limiting  problem $\Delta u-V(x_0)u+f(u)=0 \mbox{ in } \R^N$, if $\phi\in H^1(\R^N)$ satisfies $\Delta\phi-V(x_0)\phi+f^\prime(U_{x_0})\phi=0$ in $\R^N$, then $\phi=\sum_{i=1}^N a_i\frac{\partial U_{x_0}}{\partial x_i}$ for some $a_i\in \R$. Subsequently, using the Lyapunov-Schmidt method, many authors have obtained more refined results; see  \cite{ABC, AMN, DKW, KW, L, O1, O2}. 

In \cite{R},  Rabinowitz developed a variational approach, which does not require the non-degeneracy condition for limiting problems. Applying the mountain pass theorem, Rabinowitz showed  the existence of a least-energy solution of \eqref{SPP} for small $\e>0$ when 
\[
\liminf_{|x|\rightarrow\infty}V(x)>\inf_{x\in \R^N}V(x).
\]
In \cite{W}, Wang showed that the maximum point of this solution converges to the global minimum points of $V$ as $\e\rightarrow0$. In \cite{DF1},
del Pino and Felmer  proved  the existence of a single peak solution  whose maximum point converges to  potential wells as $\e\rightarrow0$. Other researchers have further developed variational methods for the construction of solutions under an optimal class of nonlinearities $f$. In \cite{BJ1, BJT, BT1}, Byeon, Jeanjean and Tanaka constructed  a single peak solution  of \eqref{SPP}  whose maximum point  converges to local minimum points or non-minimum (topologically stable) critical points of $V$ as $\e\rightarrow0$ when the nonlinearity $f$ satisfies (f1)-(f3).  In \cite{BT2}, Byeon and Tanaka proved the existence of a family of solutions of \eqref{SPP} with clustering peaks around an isolated set of critical points of $V$ that are non-minimal when the nonlinearity $f$ satisfies (f1)-(f3) as well as the following condition
\begin{enumerate}[(f4)]
  \item $f\in C^1(\R,\R)$ and there exists a $q_0\in (0,1)$ such that
  \[
  \lim_{t\rightarrow0}\frac{f(t)}{t^{1+q_0}}=0.
  \]
\end{enumerate}
We refer to \cite{ BJ1, BJ2, BJT, BT1, CJT, DPR, DF1, DF2, DF3, DF4} for further research based on  the variational argument of Rabinowitz.

In this paper, we focus on the existence of multi-bump solutions to \eqref{SPP}.  Byeon \cite{B} and  Wang \cite{W1} used the principle of symmetric criticality \cite{P} to construct  multi-bump solutions of \eqref{SPP} in which the centers of bumps converge to the sphere in $\R^N$ as $\e\rightarrow0$ when $V(x)=(|x|^2-5)^2+1$ for the first time. On the other hand, in \cite{DLY}, Dancer , Lam and Yan   showed, via the Lyapunov-Schmidt method,   the existence of multi-bump solutions of \eqref{SPP} in which the centers of bumps converge to $\mathcal{N}$ as $\e\rightarrow0$ when $f(t)=t^p$, $V(x)=dist(x,\mathcal{N})+1$ and $\mathcal{N}$ is an $(N-1)$-dimensional $C^1$ manifold without boundary. However,      the Lyapunov-Schmidt method can be used for     both a rather restricted  class of nonlinearities $f$ and   potential wells.  
Indeed,  \cite{DLY} requires not only monotone property of $f$, but also  a $C^1$ differential structure on potential wells in order to prove the existence of multi-bump solutions. Although  the set $S^{N-1}\cup (S^{N-1}+(0,\cdots,0,2))$ plays  a key role   in the Roche potential \cite{roche}, which describes the distorted shape of a star in a close binary system, these methods in \cite{B, DLY, W1} do not always guarantee the existence of multi-bump solutions to \eqref{SPP} when $f(t)=t^p$, $V(x)=dist(x,\mathcal{M})+1$ and $\mathcal{M}=S^{N-1}\cup (S^{N-1}+(0,\cdots,0,2))$ (see figure 1).  The goal of this work is to find essential  conditions on potential wells under which we can prove the existence of multi-bump solutions to \eqref{SPP}.

\begin{figure}
\subfloat{
\begingroup%
  \makeatletter%
  \providecommand\color[2][]{%
    \errmessage{(Inkscape) Color is used for the text in Inkscape, but the package 'color.sty' is not loaded}%
    \renewcommand\color[2][]{}%
  }%
  \providecommand\transparent[1]{%
    \errmessage{(Inkscape) Transparency is used (non-zero) for the text in Inkscape, but the package 'transparent.sty' is not loaded}%
    \renewcommand\transparent[1]{}%
  }%
  \providecommand\rotatebox[2]{#2}%
  \newcommand*\fsize{\dimexpr\f@size pt\relax}%
  \newcommand*\lineheight[1]{\fontsize{\fsize}{#1\fsize}\selectfont}%
  \ifx\svgwidth\undefined%
    \setlength{\unitlength}{199.23838806bp}%
    \ifx\svgscale\undefined%
      \relax%
    \else%
      \setlength{\unitlength}{\unitlength * \real{\svgscale}}%
    \fi%
  \else%
    \setlength{\unitlength}{\svgwidth}%
  \fi%
  \global\let\svgwidth\undefined%
  \global\let\svgscale\undefined%
  \makeatother%
  \begin{picture}(1,0.65903493)%
    \lineheight{1}%
    \setlength\tabcolsep{0pt}%
    \put(0,0){\includegraphics[width=\unitlength,page=1]{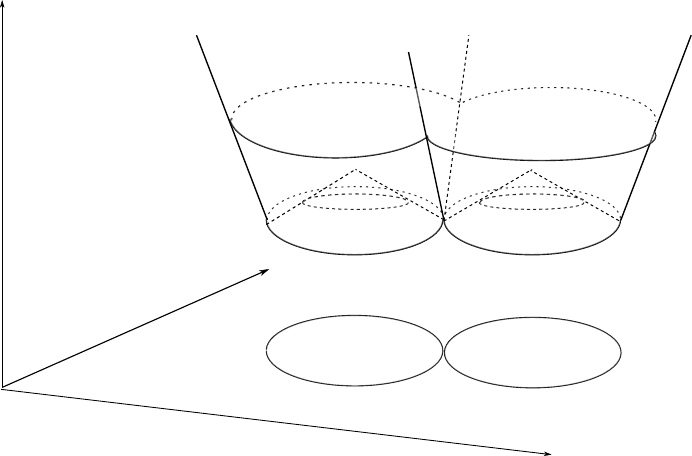}}%
    \put(0.57453817,0.04794628){\makebox(0,0)[lt]{\lineheight{1.25}\smash{\begin{tabular}[t]{l}$\mathcal{M}=S^{N-1}\cup (S^{N-1}+(0,\cdots,0,2))$\end{tabular}}}}%
    \put(0.6572766,0.632263){\makebox(0,0)[lt]{\lineheight{1.25}\smash{\begin{tabular}[t]{l}$V(x)=\operatorname{dist}(\mathcal{M},x)+1$\end{tabular}}}}%
  \end{picture}%
\endgroup%
}  
\caption{ }
\label{fig1}
\end{figure}

 In this study, we prove the existence of multi-bump solutions of \eqref{SPP} in which the centers of bumps  converge toward general sets $\mathcal{M}$ under the conditions (f1)-(f4).  
More specifically, we devise two novel topological arguments to  investigate the existence of multi-bump solutions of \eqref{SPP} in which the centers of bumps converge toward either the union of two compact smooth submanifolds of $\R^N$ where these two submanifolds meet at the origin, which is not a topological manifold;  or a topological manifold that does not admit any differentiable structure (see \cite{K}).  Extending  previous results \cite{B, DLY,W1}, our results here allow us to construct multi-bump solutions of \eqref{SPP} in which  the centers of bumps  converge toward  general classes of sets $\mathcal{M}$ for a general class of nonlinearities $f$.

Our new approach enables us to deal with general classes of sets $\mathcal{M}$, which  are defined as follows:
\begin{enumerate}[(M1)]
  \item $\mathcal{M}=\{ x\in O \ | \ V(x)=m\}$ is homeomorphic to a finite $d$-dimensional polyhedron $X$   with $H_d(X;\Z_2)\neq 0$ for some $d=1,\cdots,N-1;$
  \item  $\mathcal{M}=\{ x\in O \ | \ V(x)=m\}$ is a connected compact $d$-dimensional embedded topological submanifold of $\R^N$ for some $d=1,\cdots,N-1$.
\end{enumerate}

For the reader convenience, we recall  the definitions of  a simplicial complex, simplicial homology and an embedded topological submanifold.  A simplicial complex is a set composed of points, line segments, triangles, tetrahedrons, and their $k$-dimensional equivalents satisfying some compatibilty conditions (see  Appendix for the definition of a simplicial complex). If $K$ is a simplicial complex then we put $|K|=\bigcup\{\sigma|\sigma\in K\}$. This is called the polyhedron of $K$. We can construct a  group structure on $k$-dimensional  equivalents of $K$ (denoted by $C_k(K)$). Then we define the boundary homomorphism on these groups $C_k(K)$ as follows: specify the value of the homomorphism on each generator of $C_k(K)$; then extend linearly to the other elements, known as boundary operators $\partial_k: C_k(K;\Z_2)\rightarrow C_{k-1}(K;\Z_2)$ (see Appendix). It is well known that the composition $\partial_{k+1}\circ \partial_k$ is the zero homomorphism. Then the $k$-th homology group of $|K|$ is $H_k(|K|)=\ker \partial_k/\mbox{im}\partial_{k+1}$. An embedded topological submanifold of $M$ is a subset $S\subset M$ endowed with a topology  such that  it is a topological manifold (without boundary) and  such that    the inclusion is a topological embedding (see \cite{JML}).

 The condition that $\mathcal{M}$ has nontrivial homology is necessary in the sense that a family of multi-bump solutions cannot exist when $\mathcal{M}$ is a point or a ball in $\R^N$ (see \cite{DLY, KW}). We note that $(M1)$ and $(M2)$ are independent of each other. In fact,  $(M1)$ holds for the case $\mathcal{M}=S^{N-1}\cup (S^{N-1}+(0,\cdots,0,2))$, which  is not a topological manifold, and for $n\ge 5$ there exists a closed topological $n$-manifold that does not admit a simplicial triangulation (see \cite{M}).

 Now, our main result is discussed in below.
\begin{thm}\label{main thm}
  Suppose $N\ge 3$, $\ell_0\in \N$ and that $(V1)-(V2)$ and $(f1)-(f4)$ hold. Also assume ($M1$) or ($M2$) hold. Then for sufficiently small $\epsilon=\epsilon(\ell_0)>0$, \eqref{SPP} has a positive solution $v_\epsilon$ with exactly $\ell_0$ peaks $x_\epsilon^1, \cdots, x_\epsilon^{\ell_0}\in \R^N$ satisfying, as $\e\rightarrow 0$,
  \[
  \operatorname{dist}(x_\epsilon^j, \mathcal{M})\rightarrow 0, \ \ \frac{|x_\e^i-x_\e^j|}{\e}\rightarrow \infty,
  \]
  for $i, j=1,\ldots, \ell_0$ with $i\neq j$.
  Moreover, defining $u_\epsilon(x)=v_\epsilon(\epsilon x)$, there exist a subsequence $\epsilon_j\rightarrow0$ and a family $\{W^k\}_{k=1}^{\ell_0}\subset H^1(\R^N)$ of least energy solutions of $-\Delta W+mW=f(W), W>0$ in $\R^N$ such that
  \[
  \|u_{\epsilon_j}-\sum_{k=1}^{\ell_0}W^k(\cdot-x_{\epsilon_j}^k/\epsilon_j)\|_{H^1(\R^N)}\rightarrow0 \mbox{ as } j\rightarrow \infty.
  \]\end{thm}

We briefly compare our main step with the  one  in \cite{BT2, BT3, LS}.  In \cite{BT2, BT3, LS}, the authors obtained   an intersection property via the following topological degree lemma: 
\begin{align*}
&\mbox{ Let } M, N \mbox{ be compact } n\mbox{-dimensional } C^1 \ \mbox{manifolds without boundary with } N \mbox{ connected. }\\
&\mbox{Then homotopic maps } M\rightarrow N \mbox{ have the same mod } 2 \mbox{ degree. }
\end{align*}
However,   the condition that a topological manifold  has  a  $C^1$ atlas is essential  in this lemma  (see \cite[Theorem 1.6, Chapter 5]{H}).  In this paper, we  develop   topological arguments (as provided by Lemma \ref{y2} and Lemma \ref{y1}) to deal with general classes of sets $\mathcal{M}$, where the aforementioned atlas condition is not satisfied.

\begin{lem}\label{y2}
Let $K$ be a finite simplicial $r$-complex  with $H_r(|K|;\Z_2)\neq 0$ for some $r\in \N$, and define
$\mathcal{C}=\{f:|K|\rightarrow |K|, \mbox{ continuous }  |\ f \mbox{ is homotopic to the identity map }\}.$
 We  take a non-zero $r$-cycle $\sigma=\sum_{i=1}^k\sigma_i\in H_r(|K|;\Z_2)=\{\tau \in C_r(K; \Z_2) | \partial_r \tau=0\}$.  Then for all $f\in \mathcal{C}$, we have $\cup_{i=1}^k\big(\sigma_i(\Delta_r)\big)\subset f(|K|)$, where $\sigma_i$ is a $r$-simplex of $K$, $\Delta_r=\{x=\sum_{i=0}^r \lambda_i e_i\ |\ \sum_{i=0}^r\lambda_i=1,0\le \lambda_i\le 1\}$, and $e_0, e_1, \cdots, e_r$ are the standard basis of $\R^{r+1}$.
\end{lem}

\begin{lem}\label{y1}
  Let $M$ be a connected compact topological $n$-manifold without boundary and $f:M\rightarrow M$ be a continuous map. If $f$ is homotopic to the identity map, then $f$ is surjective.
\end{lem}
  More precisely, in obtaining an intersection property in Section \ref{an intersection property}, we construct an initial path $\gamma_0$ defined on a set $(\frac{1}{\epsilon}\mathcal{M})^{\ell_0}\times S^{\ell_0}\subset (\R^N)^{\ell_0}\times \R^{\ell_0+1}$ such that 
\[
\limsup_{\epsilon\rightarrow 0}\max\{ I_\epsilon(\gamma_0(p,s))\ |\  |p_i-p_j|\le r \mbox{ for some } i\neq j \mbox{ or } s\in \partial D^{\ell_0}_+\}<\ell_0E_m,
\] where $D^{\ell_0}_+$ denotes the closed upper hemisphere of the sphere $S^{\ell_0}$, $E_m$ denotes the least energy level of \eqref{b3}, and $I_\epsilon$ is defined in subsection \ref{fs}. We use the above property of $\gamma_0$  to define a map $
\Phi_\epsilon(\tilde{\gamma}_0(q,t)): \Big(\frac{1}{\e}\mathcal{M}\Big)^{\ell_0}\times S^{\ell_0}\rightarrow \Big(\frac{1}{\e}\mathcal{M}\Big)^{\ell_0}\times S^{\ell_0}$, where $\Phi_\epsilon(\tilde{\gamma}_0(q,t))$ is homotopic to the identity map and $\tilde{\gamma}_0$ is a path satisfying 
\[
\tilde{\gamma}_0(q,t)=\gamma_0(q,t), (q,t)\in\{(p,s)\in \Big(\frac{1}{\epsilon}\mathcal{M}\Big)^{\ell_0}\times S^{\ell_0}\ |\ |p_i-p_j|\le r \mbox{ for some } i\neq j \mbox{ or } s\in \partial D^{\ell_0}_+\}.
\]
Since $H_d(\mathcal{M};\Z_2)\neq 0$ for some $d=1,\cdots, N-1$, we take $0\neq \sum_{i=1}^k\bar{c}_i=\bar{c}\in H_d(\mathcal{M};\Z_2)$. We fix  $(\bar{q}_1,\cdots,\bar{q}_{\ell_0},N)\in c^{\ell_0}\times S^{\ell_0}\subset \mathcal{M}^{\ell_0}\times S^{\ell_0}$, where $N$ denotes the north pole of $S^{\ell_0}$, $\min_{1\le i\neq j\le \ell_0}|\bar{q}_i-\bar{q}_j|>0$, and $c=\cup_{i=1}^k \big(\bar{c}_i(\Delta_d)\big)$.
Then, by Lemma \ref{y2} and Lemma \ref{y1}, we deduce an estimate 
\[
\liminf_{\epsilon\rightarrow 0} \max_{(q,t)\in (\frac{1}{\epsilon}\mathcal{M})^{\ell_0}\times S^{\ell_0}}I_\epsilon(\tilde{\gamma}_0(q,t))\ge \ell_0E_m
\]
 for any path $\tilde{\gamma}_0$ satisfying 
\[
\tilde{\gamma}_0(q,t)=\gamma_0(q,t), (q,t)\in\{(p,s)\in \Big(\frac{1}{\epsilon}\mathcal{M}\Big)^{\ell_0}\times S^{\ell_0}\ |\ |p_i-p_j|\le r \mbox{ for some } i\neq j \mbox{ or } s\in \partial D^{\ell_0}_+\}.
\]
To the best of our knowledge, Theorem \ref{main thm} is the first result concerning the existence of multi-bump solutions of \eqref{SPP} in which the centers of bumps converge toward general classes of  potential wells.

We organize the paper as follows: In Section \ref{Preliminaries}, we  give some preliminaries for the proof of our main result. In Section \ref{A gradient estimate}, using a concentration-compactness type argument, we study a gradient estimate in an annular neighborhood of approximate solutions, which has an important role in a deformation argument. In Section \ref{an intersection property}, for our deformation argument, we study an initial path $\gamma_0(p,s):(\frac{1}{\e}\mathcal{M})^{\ell_0}\times [1-\delta_0,1+\delta_0]^{\ell_0}\rightarrow H^1(\R^N)$. We also prove, using an intersection property for a deformed path $\eta(\tau,\gamma_0(p,s))$ of the initial path $\gamma_0(p,s)$, a lower estimate of the energy level $\max_{p,s}I_\e(\eta(\tau,\gamma_0(p,s)))$ of a deformed initial path. In Section \ref{Deformation argument}, we develop a deformation argument to prove Theorem \ref{main thm}. In Appendix, we   prove Lemma \ref{y2} and Lemma \ref{y1}.

\section{Preliminaries}\label{Preliminaries}
\subsection{Notation}
For $z\in \R^N$ and $R>0$, we denote $B_R\equiv\{x\in \R^N \ | \ |x|<R\},$
\[
B(z,R)\equiv\{x\in \R^N \ | \ |x-z|<R\}\ \mbox{ and }\ \overline{B}(z,R)\equiv\{x\in \R^N \ | \ |x-z|\le R\}.
\]
For any set $E\subset \R^N$, $R>0$ and $\epsilon>0$, we define $\operatorname{dist}(x,E)=\inf_{y\in E}|x-y|$,
\[
\frac{1}{\epsilon} E \equiv \{x\in \R^N \ | \ \epsilon x\in E\}, N_R(E)\equiv \{ x\in \R^N\ | \ \operatorname{dist}(x, E)\le R\}
\]
and we denote by Int$(E)$ the interior of $E$. 
\begin{defn}
If $X$ and $Y$ are spaces, then their disjoint union $X+Y$ is  the set $X\times\{0\}\cup Y\times\{1\}$ with the topology making $X\times \{0\}$ and $Y\times \{1\}$ clopen and the inclusions $x\mapsto (x,0)$ of $X\rightarrow X+Y$ and $y\mapsto (y,1)$ of $Y\rightarrow X+Y$ homeomorphisms to their images.
\end{defn}
\begin{defn}\label{me1}
Let $X$ and $Y$ be spaces and $A\subset X$ closed. Let $f:A\rightarrow Y$ be a map. Then we denote by $Y\cup_f X$, the quotient space of the disjoint union $X+Y$ by the equivalence relation $\sim$ that is genreated by the relation $a\sim f(a)$ for $a\in A$.
\end{defn}
\subsection{Limit problems}
The following constant coefficient problem appears as a limit problem of \eqref{SPP} and play important roles in our analysis,
\begin{equation}\label{limit eq}
-\Delta u+au=f(u), u\in H^1(\R^N),
\end{equation}
where $a>0$ is a constant. In \cite{BL}, Berestycki and Lions proved that there exists a least energy solution of \eqref{limit eq} with $a=m$ if (f1)-(f3) are satisfied.  Also, they showed that each solution $U$ of \eqref{limit eq} satisfies the Pohozaev's identity
\begin{equation}\label{Pohozaev}
  \frac{N-2}{2}\int_{\R^N}|\nabla U|^2dx+N\int_{\R^N}a\frac{U^2}{2}-F(U)dx=0.
\end{equation}
In \cite{BJ1}, it was shown that the set of least energy solutions of the limit problem $-\Delta u+mu=f(u)$ in $\R^N$ is uniformly bounded in $L^\infty(\R^N)$. The bound depends only on $N, p$ and the constant $\bar{C}>0$ in (f2). Thus, by a suitable cut-off argument, we may assume that $f(t)=\bar{C}(1+t^p)$ for large $t$ without changing the set of least energy solutions. Thus, without loss of generality, we may assume that for some $\bar{C}>0$, $f$ satisfies
\[
|f(t)|+|f^\prime(t)t|\le \bar{C}(1+t^p) \mbox{ for all } t\ge 0.
\]
We define a functional corresponding to \eqref{limit eq} by
\[
L_a(u)=\frac{1}{2}\int_{\R^N}|\nabla u|^2+au^2dx-\int_{\R^N}F(u)dx,\ \ u\in H^1(\R^N).
\]
It is standard to see that $L_a\in C^2(H^1(\R^N),\R).$ By \eqref{Pohozaev}, we see  that, for each solution $U$ of \eqref{limit eq},
\begin{equation}\label{LE}
  L_a(U)=\frac{1}{N}\int_{\R^N}|\nabla U|^2 dx.
\end{equation}
We denote by $E_a=\inf\{L_a(u)\ | \ u\in H^1(\R^N)\setminus \{0\} \mbox{ and } L_a^\prime(u)=0\}.$ We set
\[
S_a=\{U\in H^1(\R^N)\setminus \{0\}\ | \ L_a^\prime(U)=0, L_a(U)=E_a, U(0)=\max_{\R^N} U(x)\}.
\]
The following result was proved in \cite[Theorem 0.2]{JT}
\begin{lem}\label{LE0}
  Let $N\ge 3$ and $a>0$.
  \begin{enumerate}[(i)]
    \item
    \[
    E_a=\inf\{ L_a(u) \ | u\in H^1(\R^N)\setminus \{0\} \mbox{ and } u \mbox{ satisfies } \eqref{Pohozaev}\},
    \]
    and the infimum is attained by $u\in H_r^1(\R^N)$, which is a least energy solution of \eqref{limit eq}.
    \item
    \[
    E_a=\inf_{\gamma \in \mathcal{P}}\max_{t\in [0,1]} L_a(\gamma(t)),
    \]
    where $\mathcal{P}=\{ \gamma(t)\in C([0,1], H_r^1(\R^N))\ | \ \gamma(0)=0, L_a(\gamma(1))<0\}.$ Moreover, the optimal path is given by $\gamma(t)=\chi(\frac{x}{Tt})$, where $\chi(x)$ is a least energy solution of \eqref{limit eq} and $T>1$ is a large constant.
  \end{enumerate}
\end{lem}
The following result was obtained in \cite[Proposition 1]{BJ1}.
\begin{lem}\label{LE1}
  For each $a>0$ and $N\ge 3$, $S_a$ is compact in $H^1(\R^N)$. Moreover, for all $U\in S_a$ there exist $C, \tilde{C}>0$, independent of $U\in S_a$ such that
  \begin{equation}\label{exp}
  U(x)+|\nabla U(x)|\le C\exp(-\tilde{C}|x|).
  \end{equation}
\end{lem}
\subsection{Euclidean neighborhood retracts}
 By Corollary \ref{iu}, we may assume that $\mathcal{M}$ is a finite $d$-dimensional  polyhedron   with $H_d(\mathcal{M};\Z_2)\neq 0$ for some $d=1,\cdots,N-1$ if $(M1)$ holds. Then if $(M1)$ holds, we take a cycle $0\neq \bar{c}=\sum_{i=1}^k \bar{c}_i\in \ker \partial_d=H_d(\mathcal{M};\Z_2)$, where $\bar{c}_i$ is a $d$-simplex in $\mathcal{M}$. We denote
 \begin{equation}\label{h4}
\cup_{i=1}^k \big(\bar{c}_i(\Delta_d)\big)=c,
 \end{equation}
where $\partial_d$ is defined in \eqref{ap} and $\cup_{i=1}^k \big(\bar{c}_i(\Delta_d)\big)$ is the support of $\sum_{i=1}^k \bar{c}_i$.  Let $10\bar{\beta}= \operatorname{dist}(\mathcal{M},\R^N\setminus O)$, where $O$ is given in (V2).  Since $\mathcal{M}$ is locally contractible, we may choose $\beta_1\in(0,\bar{\beta})$ such that   there exists
\begin{align*}
&q=(q_1,\ldots, q_{\ell_0})\in \begin{cases}
                              (c)^{\ell_0}  & \mbox{if } (M1)\mbox{ holds}, \\                   
                         (\mathcal{M})^{\ell_0} & \mbox{if } (M2) \mbox{ holds},
                          \end{cases}
\end{align*}
satisfying
\begin{equation}\label{pp2}
\min_{1\le i\neq j\le \ell_0}|q_i-q_j|=4 \beta_1.
\end{equation} 

By \cite[Appendices Theorem E.3]{Bredon}, $\mathcal{M}$ is a retract of an open neighborhood $G$ of $\mathcal{M}$ in $\R^N$. Then there exists a map
\begin{equation}\label{h3}
\omega: G\rightarrow \mathcal{M}
\end{equation}
 such that $\omega(x)=x$ for all points $x\in \mathcal{M}$. We note that  there exists  $\beta\in(0, \beta_1 )$ satisfying $N_{2\beta}(\mathcal{M})\subset G$ and
\begin{equation}\label{pil}
|\omega(x)-x|< \beta_1 \mbox{ for all } x\in N_{\frac32\beta}(\mathcal{M}).
\end{equation}

For large $R>1$ and $\rho\in (0,1]$, we try to find our critical points in the following set:
\begin{align*}
Z(\rho, R)=\Big\{\sum_{j=1}^{\ell_0} &U_j(\cdot-y_j)+w\ | \ y_j\in \frac{1}{\e}N_\beta(\mathcal{M}), U_j\in S_m, j=1,\cdots, \ell_0, \\
&|y_j-y_{j^\prime}|\ge R \mbox{ for } 1\le j\neq j^\prime\le \ell_0 \mbox{ and } \| w\|_{H^1(\R^N)}<\rho\Big\}.
\end{align*}
\subsection{Functional setting}\label{fs}
The variational framework is the following. We define a norm $\|\cdot\|_{H^1}$ on $H^1(\R^N)$ by
\[
\|u\|^2_{H^1}=\int_{\R^N}|\nabla u|^2+mu^2dx.
\]
For $u\in H^1(\R^N)$ and any set $E \subset \R^N$, we denote
\[
\|u\|_{H^1(E)}=\Big(\int_{E}|\nabla u|^2+m u^2\Big)^{1/2}.
\]
We also denote by $\|\cdot\|_{H^{-1}}$ the corresponding dual norm on $H^1(\R^N)$, that is,
\[
\|h\|_{H^{-1}}=\sup_{\|\varphi\|_{H^1}\le 1}|\langle h,\varphi\rangle|\ \ \mbox{ for }\ h\in (H^1(\R^N))^\ast=H^{-1}(\R^N),
\]
where $\langle h,\varphi \rangle$ is the duality product between $H^{-1}(\R^N)$ and $H^1(\R^N)$.
For $u\in H^1(\R^N)$, let
\begin{equation}\label{ene1}
I_\epsilon(u)=\frac{1}{2}\int_{\R^N}|\nabla u|^2+V_\epsilon(x) u^2dx-\int_{\R^N}F(u)dx,
\end{equation}
where $V_\epsilon(x)=V(\epsilon x)$.
It is standard to see that $I_\epsilon\in C^2(H^1(\R^N), \R)$ and the critical points of $I_\epsilon$ are solutions of \begin{equation}\label{main eq}
 \Delta u-V_\e(x)u+f(u)=0 \mbox{ in } \R^N.
\end{equation}
 Since we seek positive solutions, without loss of generality, we assume that $f(t)=0$ for all $t\le 0$. Then we can see from the maximum principle that any nontrivial solution of \eqref{main eq} is positive. We remark that $v(x)=u(x/\e)$ satisfies \eqref{SPP} and we try to find critical points of $I_\e(u)$.

\subsection{Estimate of $L_m(u)$ with relation to the Pohozaev identity}
Let $m>0$ be the positive number given in (V2). We define
\begin{equation}\label{p0}
P(u)=\sqrt{\Big(\frac{N\int_{\R^N}F(u)-\frac{m}{2}u^2dx}{\frac{N-2}{2}\int_{\R^N}|\nabla u|^2dx}\Big)_+}.
\end{equation}
We remark that the Pohozaev identity \eqref{Pohozaev} is $P(u)=1$ and
\[
P(u(x/s))=sP(u) \mbox{ for } u\in H^1(\R^N)\setminus \{ 0\} \mbox{ and } s>0.
\]
Then for $U\in S_m$, $s>0$ and $y\in \R^N$, we have
\[
P(U(\frac{\cdot-y}{s}))=s \mbox{ and } 
L_m(U(\frac{\cdot-y}{s}))=g(s)E_m,
\]
where $g(s)=\frac{1}{2}(Ns^{N-2}-(N-2)s^N).$
\begin{rmk}\label{rmk1}
    $g(s)$ has the following properties:
  \begin{enumerate}[(i)]
    \item $g(s)>0$ if and only if $s\in(0,\sqrt{\frac{N}{N-2}})$;
    \item $g(s)\le 1$ for all $s>0$;
    \item $g^\prime(s)>0$ for all $0<s<1$ and $g^\prime(s)<0$ for all $s>1$.
  \end{enumerate}
\end{rmk}
By Lemma \ref{LE0}, we also have
\begin{equation}\label{p1}
E_m=\inf\{ L_m(u)\ | \ u\in H^1(\R^N)\setminus \{0\} \mbox{ and } P(u)=1\}.
\end{equation}



\subsection{Local centers of mass}
We will introduce a function $\Upsilon_j(u) : Z(\rho, R)\rightarrow \R^N$ that describes local centers of mass and satisfies
\[
|\Upsilon_j(\sum_{j=1}^{\ell_0}U_j(x-y_j)+w)-y_j|\le 2R_0
\]
for some $R_0>0$ independent of $u=\sum_{j=1}^{\ell_0}U_j(x-y_j)+w\in Z(\rho, R).$

We define
\begin{equation}\label{L1}
\rho_0=\inf_{U\in S_m}\|U\|_{H^1}>0.
\end{equation}
We choose $R_0>1$ such that for all $U\in S_m$
\begin{equation}\label{L}
\|U\|_{H^1(B(0,R_0))}>\frac{3}{4}\rho_0, \ \|U\|_{H^1(\R^N\setminus\overline{B}(0,R_0))}<\frac{1}{8\ell_0}\rho_0.
\end{equation}
For $u\in H^1(\R^N)$ and $P\in \R^N$, we define
\[
d(u,P)=\psi\Big(\inf_{U\in S_m}\|u-U(x-P)\|_{H^1(B(P,R_0))}\Big),
\]
where $\psi(r)\in C_0^\infty([0,\infty),\R)$ is a function such that
\begin{align*}
&\psi(r)=\begin{cases}
         1, & r\in [0, \frac{1}{3}\rho_0], \\
         0, & r\in [\frac{1}{2}\rho_0,\infty),
       \end{cases}\\
&\psi(r)\in [0,1]\ \mbox{ for all }\ r\in [0, \infty).
\end{align*}
We set
\[
Z_0=Z(\frac{1}{8}\rho_0, 12R_0).
\]
To define local centers of mass $\Upsilon(u)=(\Upsilon_1(u),\cdots, \Upsilon_{\ell_0}(u))$, we need the following lemma. The following result was proved in \cite[Lemma 3.1]{BT2}.
\begin{lem}\label{lem1}
  There exists $r_0>0$ such that for $u(x)=\sum_{j=1}^{\ell_0}U_j(x-y_j)+w(x)\in Z_0$,
  \[
  \bigcup\limits_{j=1}^{\ell_0}\overline{B}(y_j,r_0)\subset \textrm{supp } d(u,\cdot)\subset \bigcup\limits_{j=1}^{\ell_0}\overline{B}(y_j,2R_0)
  \]
  and
  \[
  d(u,P)=1 \ \mbox{ for }\ P\in \bigcup\limits_{j=1}^{\ell_0}\overline{B}(y_j,r_0).
  \]
\end{lem}
By Lemma \ref{lem1}, for any $u\in Z_0$ there exist $\ell_0$ balls $B_j (j=1,2,\cdots, \ell_0)$ satisfying
\begin{equation}\label{p10}
\begin{cases}
\mbox{diam}B_j=5R_0 & \mbox{ for all } j\in \{1,2,\ldots, \ell_0\},\\
\operatorname{dist}(B_i,B_j)\ge 7R_0 & \mbox{ for all } 1\le i\neq j\le \ell_0,\\
d(u,\cdot)\not\equiv 0 &\mbox{ on } B_j \mbox{ for all } j\in\{1,2,\cdots,\ell_0\},\\
d(u,\cdot)=0 & \mbox{ on } \R^N \setminus \cup_{j=1}^{\ell_0}B_j.
\end{cases}
\end{equation}
For example, set $B_j=B(y_j,\frac52 R_0)$ for $u=\sum_{j=1}^{\ell_0} U_j(\cdot-y_j)+w\in Z_0$.
For $B_j$ satisfying \eqref{p10}, we define
\begin{equation}\label{ji2}
\Upsilon_j(u)=\frac{\int_{B_j}d(u,P)PdP}{\int_{B_j}d(u,P)dP}\in B_j.
\end{equation}
We note that $(\Upsilon_1(u),\ldots, \Upsilon_{\ell_0}(u))$ is uniquely determined up to a permutation and it is independent of the choice of $B_j$'s satisfying \eqref{p10}. 

We also have the following lemma. The following result was proved in \cite[Lemma 3.2]{BT2}.
\begin{lem}\label{cm}
  For $u=\sum_{j=1}^{\ell_0}U_j(x-y_j)+w(x)\in Z_0$, we have
  \begin{enumerate}[(i)]
  \item $|\Upsilon_j(u)-y_j|\le 2R_0\ (j=1, \cdots, \ell_0)$, up to a permutation.
  \item $|\Upsilon_i(u)-\Upsilon_j(u)|\ge 8R_0 \ (1\le i\neq j\le \ell_0).$
  \end{enumerate}
  Moreover, $\Upsilon_j(u)$ is uniformly locally Lipschitz continuous, that is, there exist constants $C_1, C_2>0$ such that
        \[
        |\Upsilon_j(u)-\Upsilon_j(v)|\le C_1 \|u-v\|_{H^1} \ \mbox{ for all }\ u, v\in Z_0 \ \mbox{ with }\ \|u-v\|_{H^1} \le C_2.
        \]
\end{lem}
We define $\Theta: Z_0\rightarrow \R$ by
\[
\Theta(u)=\min_{1\le i\neq j\le\ell_0}|\Upsilon_i(u)-\Upsilon_j(u)|.
\]

\subsection{A tail minimizing operator}
For a set $D\subset\R^N$, we write
\[
I_{\e, D}(u)=\frac{1}{2}\int_{D}|\nabla u|^2+V(\e x)u^2-\int_{D}F(u)dx.
\]
 For $L\ge 2$ and $z=(z_1,\cdots, z_{\ell_0})\in (\R^N)^{\ell_0}$ with $|z_i-z_j|\ge 2(L+1)$ ($1\le i\neq j\le \ell_0$), we denote
\[
D(L,z)\equiv\R^N\setminus \bigcup\limits_{j=1}^{\ell_0}\overline{B}(z_j,L).
\]
For $u\in H^1(\R^N)$, we consider the following exterior problem
\begin{equation}\label{T1}
  \begin{cases}
    -\Delta v+V(\e x)v=f(v) & \mbox{in } D(L,z), \\
    v=u & \mbox{on } \partial D(L,z).
  \end{cases}
\end{equation}
We have the following unique existence of a solution of \eqref{T1}. The proof of the following lemma is standard (refer to  \cite[Lemma 2.6]{BT2} and \cite[Proposition 5.7]{CR}).
\begin{lem}\label{Tail}
  There exist $\rho_1, \rho_2\in (0,\frac{1}{8}\rho_0]$ such that for small $\e>0$ and for $L\ge 2, z=(z_1,\cdots, z_{\ell_0})$ with $|z_i-z_j|\ge 2(L+1) (i\neq j)$ and $u\in H^1(\R^N)$ satisfying
  \begin{equation}\label{T3}
  \|u\|_{H^1(\cup_{j=1}^{\ell_0}(B(z_j,L)\setminus B(z_j,L-1)))}<\rho_1,
  \end{equation}
 the problem \eqref{T1} has a unique solution $v(x)=v_\e(L,z;u)$ in
  \[
  \{v\in H^1(D(L,z))\ | \ \|v\|_{H^1(D(L,z))}\le \rho_2\}.
  \]
  Moreover we have
  \begin{enumerate}[(i)]
  \item $v_\e(L,z;u)(x)$ is a minimizer of the following minimization problem:
  \begin{equation}\label{T4}
  \inf\{I_{\e, D(L,z)}(v)\ | \ \|v\|_{H^1(D(L,z))}\le \rho_2, v=u \mbox{ on } \partial D(L,z)\}.
  \end{equation}
  \item There exist constants $A_1, A_2, A_3>0$ independent of $\e, L\ge 2, z=(z_1,\cdots, z_{\ell_0})$ and $u$ such that
  \[
  \|v_\e(L,z;u)\|_{H^1(D(L,z))}\le A_1 \|u\|_{H^1(\cup_{j=1}^{\ell_0}(B(z_j,L)\setminus \overline{B}(z_j,L-1)))} ,
  \]
  \begin{equation}\label{T5}
  \begin{aligned}
  |v_\e(L,z;u)(x)|, |\nabla v_\e(L,z;u)(x)|\le A_2 \exp(-A_3\operatorname{dist}(x,\{z_1,\cdots z_{\ell_0}\})) \mbox{ for } x\in D(L+1,z).
  \end{aligned}
  \end{equation}
  \end{enumerate}
\end{lem}
Choosing a large $R_1\ge 12 R_0$ and a small $\rho_3\in (0,\frac{\rho_0}{8})$, we see that
\begin{equation}\label{T2}
  \|u\|_{H^1(D(\frac{R_1}{12}-1, \Upsilon(u)))}<\rho_1 \mbox{ for all } u\in Z(\rho_3,R_1).
\end{equation}
In fact, for $u=\sum_{j=1}^{\ell_0}U_j(\cdot-y_j)+w\in Z(\rho_3,R_1)$ with $U_j\in S_m$, $|y_i-y_j|\ge R_1$ for $i\neq j$, $y_j\in \frac{1}{\e}N_\beta(\mathcal{M})$  and $\|w\|_{H^1}<\rho_3$, we have
\begin{align*}
  &\|u\|_{H^1(D(\frac{R_1}{12}-1,\Upsilon(u)))}\\
  &\le \|\sum_{j=1}^{\ell_0}U_j(\cdot-y_j)\|_{H^1(D(\frac{R_1}{12}-1,\Upsilon(u)))}+\|w\|_{H^1}\\
  &\le \sum_{j=1}^{\ell_0}\|U_j\|_{H^1(\R^N\setminus B(0, \frac{R_1}{12}-1-2R_0))}+\rho_3.
\end{align*}
Here we used the fact that $|\Upsilon_j(u)-y_j|\le 2R_0$, up to a permutation. Using the uniform exponential decay(Lemma \ref{LE1}) of $U_j\in S_m$, we can show \eqref{T2} for large $R_1$ and small $\rho_3$.

We see that \eqref{T2} implies \eqref{T3} for $L\ge R_1/12$. Thus, by Lemma \ref{Tail}, the problem \eqref{T1} with $z_j=\Upsilon_j(u)$ and $L\in [\frac{R_1}{12},\frac{\Theta(u)}{3}]$ has a unique solution $v_\e(L,\Upsilon(u);u),$ where $\Upsilon(u)=(\Upsilon_1(u),\cdots, \Upsilon_{\ell_0}(u)).$ For $u\in Z(\rho_3,R_1)$ and $L\in[\frac{R_1}{12},\frac{\Theta(u)}{3}]$, we define $\tau_{\e,L}(u)\in H^1(\R^N)$ by
\[
\tau_{\e,L}(u)(x)=\begin{cases}
                    v_\e(L,\Upsilon(u);u)(x), & \mbox{for } x\in D(L,\Upsilon(u)), \\
                    u(x) &\mbox{for } x\in \cup_{j=1}^{\ell_0}B(\Upsilon_j(u),L).
                  \end{cases}
\]
 We choose $\chi(\tau)\in C^\infty(\R)$ such that $\chi(\tau)=1$ for $\tau\in (-\infty,0]$, $\chi(\tau)=0$ for $\tau\in [1,\infty)$, and $| \chi^\prime(\tau)|\le 2$. For $\xi\ge 1$, we set
\[
\chi_\xi(\tau)=\chi(\tau-\xi).
\]
We note that, since  $\Theta(u)\ge R_1-4R_0\ge \frac{2}{3}R_1$ for $u\in Z(\rho_3,R_1)$, we have $\Theta(u)/6\ge \frac{R_1}{12}$ for $u\in Z(\rho_3,R_1)$. Thus, for $j=1, \ldots, \ell_0$, we can define
\begin{align*}
&\tau_\e(u)(x)=\tau_{\e,\frac{\Theta(u)}{6}}(u)(x),\\
&\hat{\tau}_{\e,j}(u)(x)=\chi_{\Theta(u)/3}(|x-\Upsilon_j(u)|)\tau_\e(u)(x).
\end{align*}
We have the following properties (refer to \cite[Proposition 4.4]{BT2}, \cite[Lemma 3.4]{BT3}  and \cite[Proposition 5.24]{CR}).
\begin{lem}\label{T6}
  \begin{enumerate}[(i)]
    \item $I_\e(\tau_\e(u))\le I_\e(u)$ for all $u\in Z(\rho_3,R_1)$.
    \item There exist $A_4, A_5>0$ such that
    \[
    \Big\|\sum_{j=1}^{\ell_0}\hat{\tau}_{\e,j}(u)-\tau_\e(u)\Big\|_{H^1}\le A_4\exp(-A_5\Theta(u)) \mbox{ for } u\in Z(\rho_3, R_1).
    \]
  \end{enumerate}
\end{lem}

\section{A gradient estimate for the energy functional}\label{A gradient estimate}
\subsection{$\e$-dependent concentration-compactness type result}
To show the existence of multi-bump solutions of \eqref{SPP} for small $\epsilon>0$, we need the following concentration-compactness type result. The following result was proved in \cite[Proposition 5.1]{BT2}.
\begin{prop}\label{c}
  Suppose that a bounded sequence $(u_{\epsilon_n})_{n=1}^\infty\subset H^1(\R^N)$ satisfies for $b\in \R$ and an open bounded set $A\subset \R^N$
  \begin{align*}
    &\epsilon_n \rightarrow 0,\\
    &I_{\epsilon_n}(u_{\epsilon_n})\rightarrow b,\\
    &\| I_{\epsilon_n}^\prime(u_{\epsilon_n})\|_{H^{-1}}\rightarrow 0,\\
    &\|u_{\epsilon_n}\|_{H^1(\R^N\setminus \frac{1}{\epsilon_n}A)}\rightarrow 0 \mbox{ as } n\rightarrow \infty.
  \end{align*}
  Then there exist $\ell \in \N\cup \{0\}, (z_{\epsilon_n}^k)\subset \R^N, z_0^k\in \overline{A}, W^k\in H^1(\R^N)$ for $k=1,\cdots, \ell$ such that as $n\rightarrow \infty$ (after extracting a subsequence if necessary)
  \begin{align}
    &\epsilon_nz_{\epsilon_n}^k\rightarrow z_0^k\in \overline{A} \mbox{ for } k=1,\cdots, \ell;\label{c1} \\
    &|z_{\epsilon_n}^k-z_{\epsilon_n}^{k^\prime}|\rightarrow \infty \mbox{ for } 1\le k \neq k^\prime \le \ell;\label{c2}\\
    &u_{\epsilon_n}(x+z_{\epsilon_n}^k)\rightharpoonup W^k(x) \mbox{ weakly in } H^1(\R^N) \mbox{ for } k=1,\cdots, \ell;\label{c3}\\
    &W^k \mbox{ is a positive solution of } -\Delta W+V(z_0^k)W=f(W) \mbox{ in } \R^N;\label{c4}\\
    &\|u_{\epsilon_n}(x)-\sum_{k=1}^{\ell}W^k(x-z_{\epsilon_n}^k)\|_{H^1}\rightarrow 0;\label{c5}\\
    &I_{\epsilon_n}(u_{\epsilon_n})\rightarrow \sum_{k=1}^{\ell}L_{V(z_0^k)}(W^k) \label{c6}.
  \end{align}
\begin{rmk}
When $\ell=0$, the claim of Proposition \ref{c} is
\[
\|u_{\e_n}\|_{H^1}\rightarrow 0 \mbox{ and } I_{\e_n}(u_{\e_n})\rightarrow 0.
\]
\end{rmk}
\end{prop}
\subsection{A gradient estimate}
We define $\hat{\rho}: Z_0\rightarrow \R$ by
\[
\hat{\rho}(u)=\inf_{U_j\in S_m, y_j\in \frac{1}{\e}N_\beta(\mathcal{M})}\|u-\sum_{j=1}^{\ell_0}U_j(x-y_j)\|_{H^1}.
\]
The following proposition  gives a uniform
estimate of $\|I^\prime_\e(\cdot)\|_{H^{-1}}$ in an annular neighborhood of a set of expected solutions.
\begin{prop}\label{g1}
  Let $b_\epsilon$ be a sequence satisfying $b_\epsilon\rightarrow\ell_0 E_m$ as $\epsilon\rightarrow0$. Then for any $0<\rho<\rho^\prime<\frac{\rho_0}{8}$ and $12R_0<r^\prime<r$, there exists $\nu_0>0$ such that for small $\epsilon>0$
  \[
  \|I_\epsilon^\prime(u)\|_{H^{-1}}\ge \nu_0
  \]
  for all $u\in Z(\rho^\prime,r^\prime)\setminus Z(\rho,r)$ with $I_\epsilon(u)\le b_\epsilon$.
\end{prop}
\begin{proof}
  It suffices to show that if there exist $\e_n>0$ and $u_{\e_n}$ satisfying
  \begin{align}
    &\e_n\rightarrow 0,\label{g3}\\
    &u_{\e_n}\in Z(\rho^\prime,r^\prime),\label{g4}\\
    &I_{\e_n}(u_{\e_n})\le b_{\e_n},\label{g5}\\
    &\|I_{\e_n}^\prime(u_{\e_n})\|_{H^{-1}}\rightarrow0\label{g6},
  \end{align}
  then
  \begin{align}
    &\hat{\rho}(u_{\e_n})\rightarrow 0,\label{g7}\\
    &\Theta(u_{\e_n})\rightarrow\infty \mbox{ as } n\rightarrow\infty.\label{g8}
    \end{align}
  The behavior \eqref{g7}-\eqref{g8} imply $u_{\e_n}\in Z(\rho,r)$ for large $n$.

  Based on assumptions \eqref{g3}-\eqref{g6}, we study the behavior of $u_{\e_n}$ as $n\rightarrow\infty$. Note that \eqref{g4} implies that $(u_{\e_n})$ is a bounded sequence in $H^1(\R^N)$. For the sake of simplicity of notation, we write $\e$ for $\e_n$.  We divide the proof into several steps.\\
  {\bf Step 1.} For $u_\e(x)=\sum_{j=1}^{\ell_0}U_j(x-y_j)+w(x) \in Z(\rho^\prime, r^\prime),$ we have $\|u_\e\|_{H^1(\R^N\setminus \cup_{j=1}^{\ell_0}\overline{B}(y_j,R_0))}< \frac{1}{4}\rho_0.$\\
  By \eqref{g4}, we may write $u_\e(x)=\sum_{j=1}^{\ell_0}U_j(x-y_j)+w(x),$ where $U_j\in S_m$, $|y_i-y_j|\ge r^\prime$ for $i\neq j$, $y_j\in \frac{1}{\e}N_\beta(\mathcal{M})$ and $\|w\|_{H^1}\le \rho^\prime$. By \eqref{L}, it follows that
  \begin{align*}
    &\|u_\e\|_{H^1(\R^N\setminus \cup_{j=1}^{\ell_0}\overline{B}(y_j,R_0))}\\
    &\le \|\sum_{j=1}^{\ell_0}U_j(x-y_j)\|_{H^1(\R^N\setminus \cup_{j=1}^{\ell_0}\overline{B}(y_j,R_0))}+\rho^\prime\\
    &\le \sum_{j=1}^{\ell_0}\|U_j(x-y_j)\|_{H^1(\R^N\setminus \overline{B}(y_j,R_0))}+\rho^\prime\\
    &<\frac{1}{8}\rho_0+\frac{1}{8}\rho_0=\frac{1}{4}\rho_0.
  \end{align*}
  {\bf Step 2.} $\|u_\e\|_{H^1(\R^N\setminus \frac{1}{\e}N_{2\beta}(\mathcal{M}))}\rightarrow 0$ as $\e \rightarrow0$.\\
  By Step 1 and the facts $\overline{B}(y_j,R_0)\subset\frac{1}{\e}N_{\frac{3}{2}\beta}(\mathcal{M})$ for all $j=1,\ldots, \ell_0$, we have
  \begin{equation}\label{g10}
  \|u_\e\|_{H^1(\R^N\setminus \frac{1}{\e}N_{\frac{3}{2}\beta}(\mathcal{M}))}<\frac{\rho_0}{4}.
  \end{equation}
  Setting $n_\e=[\frac{\beta}{2\e}]$, we get
  \[
  \sum_{k=1}^{n_\e}\|u_\e\|_{H^1(\frac{1}{\e}(N_{\frac{3}{2}\beta+\e k}(\mathcal{M})\setminus N_{\frac{3}{2}\beta+\e (k-1)}(\mathcal{M})))}^2\le \|u_\e\|_{H^1(\R^N\setminus \frac{1}{\e}N_{\frac{3}{2}\beta}(\mathcal{M}))}^2<\frac{1}{16}\rho_0^2.
  \]
  Then we can find $k_\e\in \{1,\ldots, n_\e\}$ such that
  \begin{equation}\label{g9}
  \|u_\e\|_{H^1(\frac{1}{\e}(N_{\frac{3}{2}\beta+\e k_\e}(\mathcal{M})\setminus N_{\frac{3}{2}\beta+\e (k_\e-1)}(\mathcal{M})))}\rightarrow0 \mbox{ as } \e \rightarrow 0.
  \end{equation}
  We take a function $\varphi_\e(x)\in C^\infty(\R^N)$ such that
  \begin{equation}\label{pp20}
  \varphi_\e(x)=\begin{cases}
                  1, & \mbox{for } x\in \R^N\setminus\frac{1}{\e}N_{\frac{3}{2}\beta+\e (k_\e-\frac14)}(\mathcal{M}), \\
                  0, & \mbox{for } x\in \frac{1}{\e}N_{\frac{3}{2}\beta+\e (k_\e-\frac34))}(\mathcal{M}),
                \end{cases}
                |\nabla \varphi_\e(x)|\le C \mbox{ for } x\in \R^N,
  \end{equation}
  where $C>0$ is independent of $\e$. We note that, by (f1)-(f3), for any $a>0$ there exists $C_a>0$ such that
   \[
   |f(x)|\le a|x|+C_a|x|^p \mbox{ for all } x \in \R.
   \]
   By computing $I^\prime_\e(u_\e)(\varphi_\e u_\e)$ and using \eqref{g9}, we see that, for $a_0\in (0,\frac{1}{2}V_0),$
  \begin{align*}
    I_\e^\prime(u_\e)(\varphi_\e u_\e)&=\int_{\R^N}\nabla u_\e\cdot \nabla(\varphi_\e u_\e)+V(\e x)u_\e^2\varphi_\e-f(u_\e)\varphi_\e u_\e dx\\
    &=\int_{\R^N\setminus\frac{1}{\e}N_{\frac{3}{2}\beta+\e (k_\e-\frac14)}(\mathcal{M})}|\nabla u_\e|^2+V(\e x)u_\e^2-f(u_\e)u_\e dx+o(1)\\
    &\ge\int_{\R^N\setminus\frac{1}{\e}N_{\frac{3}{2}\beta+\e (k_\e-\frac14)}(\mathcal{M})}|\nabla u_\e|^2+V_0u_\e^2-a_0u_\e^2-C_{a_0}u_\e^{p+1}dx+o(1)\\
    &\ge\int_{\R^N\setminus\frac{1}{\e}N_{\frac{3}{2}\beta+\e (k_\e-\frac14)}(\mathcal{M})}|\nabla u_\e|^2+\frac{V_0}{2}u_\e^2-C_{a_0}u_\e^{p+1}dx+o(1)\\
& =\int_{\R^N\setminus A_\e}|\nabla u_\e|^2+\frac{V_0}{2}u_\e^2-C_{a_0}u_\e^{p+1}dx+o(1)\mbox{ as } \e \rightarrow 0,
  \end{align*}
where $A_\e$ is a smooth domain satisfying $\frac{1}{\e}N_{\frac{3}{2}\beta+\e (k_\e-\frac14)}(\mathcal{M})\subset A_\e \subset \frac{1}{\e} N_{\frac{3}{2}\beta+\e k_\e}(\mathcal{M})$.  Then, by \eqref{g10}, if we choose $\rho_0$ small enough, there exists $\tilde{c}>0$ such that
  \begin{equation}\label{a2}
   I_\e^\prime(u_\e)(\varphi_\e u_\e) \ge \tilde{c}\|u_\e\|^2_{H^1(\R^N\setminus\frac{1}{\e}N_{\frac{3}{2}\beta+\e k_\e}(\mathcal{M}))} \mbox{ for small } \e>0.
  \end{equation}
 Thus, by \eqref{g6}, \eqref{pp20} and \eqref{a2}, since $\{u_\e\}$ is a bounded sequence in $H^1(\R^N)$, we get
  \[
  \|u_\e\|_{H^1(\R^N\setminus \frac{1}{\e}N_{2\beta}(\mathcal{M}))}\le \|u_\e\|_{H^1(\R^N\setminus\frac{1}{\e}N_{\frac{3}{2}\beta+\e k_\e}(\mathcal{M}))}\rightarrow 0 \mbox{ as } \e\rightarrow 0.
  \]
  {\bf Step 3.} There exist $\ell\in \N, (z_\e^k)\subset \R^N, z_0^k\in N_{2\beta}(\mathcal{M}), W^k\in H^1(\R^N)$ for $k=1, \ldots, \ell$ such that \eqref{c1}-\eqref{c6} hold after extracting a subsequence if necessary.\\
 By applying Proposition \ref{c} to $A=Int(N_{2\beta}(\mathcal{M}))$, we can show Step 3.
   We remark that $(z_\e^k)\subset\R^N$ are found as points that satisfies
  \begin{equation}\label{a1}
    \int_{B(z_\e^k,1)}|u_\e|^2dx\nrightarrow0,
  \end{equation}
  $u_\e(x+z_\e^k)\rightharpoonup W^k(x)$ weakly in $H^1(\R^N)$ and strongly in $H^1_{loc}(\R^N)$ as $\e \rightarrow 0$, after extracting a subsequence.\\
{\bf Step 4.} $(u_\e)$ satisfies $\Theta(u_\e)\rightarrow\infty$ as $\e \rightarrow 0$.\\
 By \eqref{g4}, we may write
 \[
 u_\e(x)=\sum_{j=1}^{\ell_0}U_j(x-y_j)+w(x),
 \]
  where $U_j\in S_m, y_j\in \frac1\e N_{\beta}(\mathcal{M})$,  $\min_{1\le i\neq j\le \ell_0}|y_i-y_j|\ge r^\prime$ and $\|w\|_{H^1}\le \rho^\prime$. We fix $i_0\in \{1,\ldots, \ell_0\}$ arbitrarily. We remark that \eqref{a1} holds for $y_{i_0}$ if  $r^\prime$ is large and  $\rho^\prime$ is small. Suppose that there exists an index $i\neq i_0$ such that $|y_i-y_{i_0}|$ stays bounded as $\e\rightarrow0$. We denote by $\{i_1, \ldots, i_k\}$ the set of such indices and let
 \[
 \overline{R}=\limsup_{\e\rightarrow0}\max_{0\le j\neq j^\prime\le k}|y_{i_j}-y_{i_{j^\prime}}|\in [r^\prime, \infty).
 \]
 Then we can see that the weak limit $W^{i_0}=\lim_{\e\rightarrow0}u_\e(x+y_{i_0})$ satisfies $W^{i_0}(x)\in K$, where
 \[
 K=\{\sum_{j=0}^{k}U_j(x-y_j)+w\ | \ y_j\in \R^N, |y_j-y_{j^\prime}|\in [r^\prime, \overline{R}] \mbox{ for } j\neq j^\prime, U_j\in S_m, \|w\|_{H^1}\le \rho^\prime\}.
 \]
 We claim that any $u=\sum_{j=0}^{k}U_j(x-y_j)+w\in K$ is not radially symmetric. In fact, by \eqref{L} and the fact that $|y_i-y_j|\ge 12R_0$ for $i\neq j$, it follows that
 \begin{align*}
   &\|\sum_{i=0}^{k}U_i(x-y_i)+w\|_{H^1(B(y_j,R_0))}\\
   &\ge \|U_j\|_{H^1(B(0,R_0))}-\sum_{i\in \{0,\ldots, k\}\setminus\{j\}}\|U_i\|_{H^1(\R^N\setminus B(0,R_0))}-\|w\|_{H^1}\\
   &>\frac{3}{4}\rho_0-\frac{k}{8\ell_0}\rho_0-\frac{1}{8}\rho_0>\frac{1}{2}\rho_0 \mbox{ for any } j\in \{0,\ldots, k\},\\
   &\|\sum_{i=0}^{k}U_i(x-y_i)+w\|_{H^1(\R^N\setminus \cup_{j=1}^k B(y_j,R_0))}\\
   &\le \sum_{i=0}^{k}\|U_i\|_{H^1(\R^N\setminus B(0,R_0))}+\|w\|_{H^1}\\
   &<\frac{k+1}{8\ell_0}\rho_0+\frac{1}{8}\rho_0\le\frac{1}{4}\rho_0.
 \end{align*}
 This implies that any $u\in K$ is not radially symmetric even up to a translation. In particular, $W^{i_0}$ is not radially symmetric even up to a transration. On the other hand, by Step 3, $W^{i_0}(x)$ is a positive solution of $-\Delta W+ V(z_0^{i_0})W=f(W)$ with $z_0^{i_0}=\lim_{\e\rightarrow0}\e y_{i_0}$, and then, by the result of Gidas-Ni-Nirenberg \cite{GNN}, $W^{i_0}$ must be radially symmetric up to a translation. It is a contradiction and we have $\Theta(u_\e)\ge \min_{1\le i\neq j\le \ell_0}|y_i-y_j|-4R_0\rightarrow\infty$ as $\e\rightarrow0$. Thus, the conclusion of Step 4 holds.\\
{\bf Step 5.} $\hat{\rho}(u_\e)\rightarrow 0.$\\
By Step 4, we can take $z_\e^k=\Upsilon_k(u_\e)$ for $k=1,\ldots, \ell_0$  in Step 3 and we see that the constant $\ell$ in Step 3 satisfies $\ell \ge\ell_0$. We claim $\ell=\ell_0$. In fact, it follows from \eqref{c6} that
\[
\sum_{k=1}^{\ell}L_{V(z_0^k)}(W^k)=\lim_{\e\rightarrow0}I_\e(u_\e)\le \lim_{\e\rightarrow 0}b_\e\le \ell_0E_m.
\]
By (V2) and the fact that $z_0^k\in N_{2\beta}(\mathcal{M}),$ recalling from \cite{JT} that $E_a>E_b$ if $a>b$, , we have $L_{V(z_0^k)}(W^k)\ge E_{V(z_0^k)}\ge E_m$. This implies $\ell=\ell_0$. We remark that $z_0^k\in \mathcal{M}$ and $W^k(x+h_k)\in S_m$ for some $h_k\in \R^N$. It follows that
\[
\hat{\rho}(u_\e)\le \|u_\e(x)-\sum_{k=1}^{\ell_0}W^k(x+h_k-(z_\e^k+h_k))\|_{H^1}\rightarrow0,
\]
as $\e\rightarrow 0$, which completes the proof of Proposition \ref{g1}.
\end{proof}

For a proof of our main theorem, we also need the following compactness result. The following result was proved in \cite[Proposition 5.6]{BT2}.
\begin{prop}\label{g2}
  For a fixed $\e\in(0,1]$, suppose that for some $b\in \R$, there exists a sequence $\{u_j\}_{j=1}^\infty\subset H^1(\R^N)$ satisfying
  \begin{align*}
    &u_j\in Z_0,\\
    &\|I_\e^\prime(u_j)\|_{H^{-1}}\rightarrow0,\\
    &I_\e(u_j)\rightarrow b \mbox{ as } j\rightarrow\infty.
  \end{align*}
  Then $b$ is a critical value of $I_\e$ and the sequence $\{u_j\}_{j=1}^\infty\subset H^1(\R^N)$ has a strongly convergent subsequence in $H^1(\R^N)$.
\end{prop}

\section{An initial path and an intersection property}\label{an intersection property}
To find a critical point of $I_\e$, we will use a deformation argument. In this section, we choose an initial path for a deformation argument.
\subsection{An initial path}
We choose and fix a least energy solution $U_0\in S_m$ of the limit problem \eqref{limit eq} with $a=m$. For $p=(p_1,\cdots, p_{\ell_0})\in (\frac1\e\mathcal{M})^{\ell_0}$ and $s=(s_1,\cdots ,s_{\ell_0})\in (0,\infty)^{\ell_0},$ we set
\[
\gamma_0(p,s)(x)=\sum_{j=1}^{\ell_0}U_0(\frac{x-p_j}{s_j})
\]
and
\begin{equation}\label{llo1}
\rho_4=\frac{1}{2}\min\{\frac{1}{8}\rho_0, \rho_3\}>0,
\end{equation}
where $\rho_0, \rho_3$ are given in \eqref{L1}, \eqref{T2}. Now we choose $\delta_0>0$ small such that
\begin{equation}\label{i4}
\|\gamma_0(p,s)-\sum_{j=1}^{\ell_0}U_0(x-p_j)\|_{H^1}<\frac{\rho_4}{2}
\end{equation}
for all $p=(p_1,\cdots, p_{\ell_0})\in (\frac1\e\mathcal{M})^{\ell_0}$ and $s=(s_1,\cdots ,s_{\ell_0})\in [1-\delta_0, 1+\delta_0]^{\ell_0}.$
We set
\[
\xi(p)= \xi(p_1, \ldots, p_{\ell_0})=\min_{1\le i\neq j\le \ell_0} |p_i-p_j|.
\]
The following lemma was proved in \cite[Lemma 8.1]{BT2}.
\begin{lem}\label{i0}
For all $p=(p_1,\cdots, p_{\ell_0})\in (\frac1\e\mathcal{M})^{\ell_0}$ with $|p_i-p_j|\ge R_1 (i\neq j)$ and $s=(s_1,\cdots ,s_{\ell_0})\in [1-\delta_0, 1+\delta_0]^{\ell_0}$, there exist constants $C_1^\prime, C_2^\prime>0$ such that
\begin{enumerate}[(i)]
  \item $\| \gamma_0(p,s)-U_0(\frac{x-p_j}{s_j})\|_{H^1(B(p_j,\frac{\xi(p)}{3}))}\le C_1^\prime e^{-C_2^\prime \xi(p)} \mbox{ for } j=1,\ldots , \ell_0;$
  \item $\|\gamma_0(p,s)\|_{H^1(\R^N \setminus \cup_{j=1}^{\ell_0} B(p_j,\frac{\xi(p)}{3}))}\le C_1^\prime e^{-C_2^\prime \xi(p)};$
  \item $|L_m(\gamma_0(p,s))-(\sum_{j=1}^{\ell_0}g(s_j))E_m|\le C_1^\prime e^{-C_2^\prime \xi(p)}.$
\end{enumerate}
\end{lem}
We set for $r\ge R_1$
\[
A_\e(r)=\Big\{ p=(p_1,\cdots,p_{\ell_0})\in \Big(\frac1\e\mathcal{M}\Big)^{\ell_0}\ \Big|\ |p_i-p_j|\ge r (i\neq j)\Big\}\times [1-\delta_0, 1+\delta_0]^{\ell_0}.
\]
We note that for any $r>0$, $A_\e(r)$ is not empty if $\e>0$ is sufficiently small and
\[
\partial A_\e(r)=\{(p,s)\in A_\e(r)\ | \ |p_i-p_j|=r \mbox{ for some } i\neq j \mbox{ or } s_j\in \{1-\delta_0, 1+\delta_0\} \mbox{ for some } j\}.
\]
The following result was proved in \cite[Proposition 10.1]{BT2}.
\begin{prop}\label{it}
Let $U\in S_m$. Then  for any fixed $\bar{\delta}\in(0,1)$, there exists $R_2\in (R_1,\infty)$ with the following property: for any $r\ge R_2$, there exists a constant $C(r)>0$ such that
  \[
  L_m(\sum_{j=1}^{\ell_0}U(\frac{x-p_j}{s_j}))\le \ell_0 E_m-C(r)
  \]
  for any $(p,s)\in (\R^N)^{\ell_0}\times [1-\bar{\delta},1+\bar{\delta}]^{\ell_0}$ with $\xi(p)=r$.
\end{prop}
The following proposition is one of keys of our argument.
\begin{prop}\label{inter}
  There exists $R_3\in (R_2,\infty)$ with the following property: for any $r\ge R_3$, there exists $\nu_1>0$ such that
  \begin{equation}\label{i6}
  \limsup_{\e\rightarrow0}\max_{(p,s)\in \partial A_\e(r)} I_\e(\gamma_0(p,s))\le \ell_0 E_m-\nu_1.
  \end{equation}
\end{prop}
\begin{proof}
We note that, from Lemma \ref{LE1},
\begin{equation}\label{i5}
\begin{aligned}
I_\e(\gamma_0(p,s))&=L_m(\gamma_0(p,s))+\frac{1}{2}\int_{\R^N}(V(\e x)-m)(\gamma_0(p,s))^2dx\\
&=L_m(\gamma_0(p,s))+\frac{1}{2}\int_{\cup_{j=1}^{\ell_0}B(p_j,\frac{1}{\sqrt{\e}})}(V(\e x)-m)(\gamma_0(p,s))^2dx\\
&\ \ +\frac{1}{2}\int_{\R^N\setminus\cup_{j=1}^{\ell_0}B(p_j,\frac{1}{\sqrt{\e}})}(V(\e x)-m)(\gamma_0(p,s))^2dx\\
                   &=L_m(\gamma_0(p,s))+o(1) \mbox{ as } \e\rightarrow0.
\end{aligned}
\end{equation}
  For $(p,s)\in \partial A_\e(r)$, we have either
  \begin{enumerate}[(a)]
    \item $|p_i-p_j|=r \mbox{ for some } i\neq j$ or
    \item $s_j\in \{1-\delta_0, 1+\delta_0\}$ for some $j$.
  \end{enumerate}
  When (a) takes a place, from  Proposition \ref{it} and \eqref{i5}, \eqref{i6} holds for small $\e$ and large $r$.
  When (b) takes a place, from Lemma \ref{i0} (iii) and \eqref{i5}, we have
  \[
  I_\e(\gamma_0(p,s))\le (\sum_{j=1}^{\ell_0}g(s_j))E_m+C_1^\prime e^{-C_2^\prime \xi(p)}+o(1) \mbox{ as } \e\rightarrow0.
  \]
  Thus, choosing $\e>0$ smaller and $r>0$ larger, from Remark \ref{rmk1}, we get \eqref{i6}.
\end{proof}

Next, we show the following estimate.
\begin{lem}\label{i2}
For any $r\ge R_3$, we have
\begin{equation}\label{pl}
  \lim_{\e\rightarrow0}\max_{(p,s)\in A_\e(r)}I_\e(\gamma_0(p,s))=\ell_0 E_m.
\end{equation}
\end{lem}
\begin{proof}
Let $b_\e =\max_{(p,s)\in A_\e(r)}I_\e(\gamma_0(p,s))=I_\e(\gamma_0(p_\e,s_\e))$ with $p_\e=(p_1^\e, \ldots, p_{\ell_0}^\e)$. We claim that
\begin{equation}\label{i1}
\min_{1\le i\neq j \le \ell_0}|p_i^\e-p_j^\e|=\infty \mbox{ as } \e\rightarrow0.
\end{equation}
In fact, suppose $\liminf_{\e\rightarrow 0} \min_{1\le i\neq j \le \ell_0}|p_i^\e-p_j^\e|=\bar{r}\in[r,\infty)$. Then, by Proposition \ref{it}, as in \eqref{i5}, we have
\[
\liminf_{\e\rightarrow 0} I_\e(\gamma_0(p_\e, s_\e))\le \ell_0E_m-C(\bar{r}).
\]
On the other hand, from Lemma \ref{i0} (iii), for  $q_\e=(q_1^\e, \ldots, q_{\ell_0}^\e)\in \Big(\frac1\e\mathcal{M}\Big)^{\ell_0}$ with $\min_{i\neq j}|q_i^\e-q_j^\e|\ge \frac{c_1}{\e}$, where  $c_1$ is a positive constant independent of $\e$,  we have
\begin{align*}
b_\e &\ge I_\e(\gamma_0(q_\e, (1,\ldots,1)))\\
&= I_\e(\sum_{j=1}^{\ell_0}U_0(x-q_j^\e))\\
&=L_m(\sum_{j=1}^{\ell_0}U_0(x-q_j^\e))+\frac{1}{2}\int_{\R^N}(V(\e x)-m)(\sum_{j=1}^{\ell_0}U_0(x-q_j^\e))^2\\
&=\ell_0 E_m+o(1) \mbox{ as } \e \rightarrow0,
\end{align*}
which is a contradiction and we have \eqref{i1}.
For $p_1\in \frac1\e\mathcal{M}, s_1\in [1-\delta_0, 1+\delta_0]$,
  \[
  I_\e(U_0(\frac{x-p_1}{s_1}))\rightarrow L_m(U_0(\frac{x}{s_1}))=g(s_1)E_m \mbox{ as } \e\rightarrow 0.
  \]
  We also note that $\max_{s_1\in [1-\delta_0, 1+\delta_0]}g(s_1)E_m=E_m$. Since $U_0(x)$ decays exponentially as $|x|\rightarrow \infty$, from Lemma \ref{i0} (iii) and \eqref{i1}, we can deduce \eqref{pl}.
\end{proof}
We have the following lemma.
\begin{lem}\label{i3}
  \begin{enumerate}[(i)]
    \item Up to a permutation, we have for small $\e>0$ and $r\ge R_1$,
    \[
    |\Upsilon_j(\gamma_0(p,s))-p_j|\le 2R_0 \mbox{ for all } (p,s)\in A_\e(r) \mbox{ and }
    \]
    \item
    \begin{equation*}
    \lim_{r\rightarrow\infty}\limsup_{\e\rightarrow0}\max_{(p,s)\in A_\e(r), j=1,\ldots,\ell_0}|P(\hat{\tau}_{\e,j}(\gamma_0(p,s)))-s_j|=0.
    \end{equation*}
  \end{enumerate}
\end{lem}
\begin{proof}
  (i) follows from Lemma \ref{cm} (i) and \eqref{i4}.\\
  (ii) By Lemma \ref{LE1}, we have
  \[
  \lim_{r\rightarrow \infty}\limsup_{\e\rightarrow0}\max_{(p,s)\in A_\e(r)}\|\tau_\e(\gamma_0(p,s))-\gamma_0(p,s)\|_{H^1}=0.
  \]
  Since $\gamma_0(p,s)\in Z(\frac{1}{2}\rho_4, r)$ for all $(p,s)\in A_\e(r)$, by (ii) of Lemma \ref{T6}, we have
  \[
  \lim_{r\rightarrow \infty}\limsup_{\e\rightarrow0}\max_{(p,s)\in A_\e(r)} \Big\|\sum_{j=1}^{\ell_0}\hat{\tau}_{\e,j}(\gamma_0(p,s))-\tau_\e(\gamma_0(p,s))\Big\|_{H^1}=0.
  \]
  Then we have
  \[
  \lim_{r\rightarrow \infty}\limsup_{\e\rightarrow0}\max_{(p,s)\in A_\e(r)} \Big\|\sum_{j=1}^{\ell_0}\hat{\tau}_{\e,j}(\gamma_0(p,s))-\gamma_0(p,s)\Big\|_{H^1}=0.
  \]
  Since
  \[
  \lim_{r\rightarrow \infty}\limsup_{\e\rightarrow0}\max_{(p,s)\in A_\e(r), j=1,\ldots,\ell_0} \Big\|\chi_{\Theta(\gamma_0(p,s))/3}(|\cdot-\Upsilon_j(\gamma_0(p,s))|)\gamma_0(p,s)-U_0(\frac{\cdot-p_j}{s_j})\Big\|_{H^1}=0,
  \]
 up to a permutation, we have (ii).
\end{proof}
By Lemma \ref{i3}, we see that there exists $\e_0>0$ with the following property: for $0<\e<\e_0$, up to a permutation,
  \begin{equation}\label{h2}
  \max_{(p,s)\in A_\e(kR_3), j=1,\ldots, \ell_0}|\Upsilon_j(\gamma_0(p,s))-p_j|\le 2R_0
  \end{equation}
  and
\begin{equation}\label{h1}
  \max_{(p,s)\in A_\e(kR_3), j=1,\ldots, \ell_0}|P(\hat{\tau}_{\e, j}(\gamma_0(p,s)))-s_j|<\frac{\delta_0}{8},
  \end{equation}
where $k$ is a large constant with $k\ge 3$.
By \eqref{h2}, \eqref{h1} and the fact that $\{\gamma_0(p,s)\ |\ (p,s)\in (\frac{1}{\e}\mathcal{M})^{\ell_0}\times [1-\delta_0,1+\delta_0]^{\ell_0}\}$ is compact, we extend the center of mass $\Upsilon_j$ and $P\circ\hat{\tau}_{\e,j}$ on $ \{\gamma_0(p,s)\ |\ (p,s)\in A_\e(kR_3)\}$ to functions $\hat{\Upsilon}_j$ and $\hat{P}_j$ on
\begin{equation}\label{ni1}
 \{\gamma_0(p,s)\ |\ (p,s)\in \Big(\frac{1}{\e}\mathcal{M}\Big)^{\ell_0}\times [1-\delta_0,1+\delta_0]^{\ell_0},\}
\end{equation}
respectively, so that,  up to a permutation,
\begin{equation}\label{s1}
\begin{aligned}
  &\max_{(p,s)\in (\frac{1}{\e}\mathcal{M})^{\ell_0}\times [1-\delta_0,1+\delta_0]^{\ell_0}, j=1,\ldots, \ell_0}|\hat{\Upsilon}_j(\gamma_0(p,s))-p_j|\le 3R_0,\\
&\ \ \max_{(p,s)\in (\frac{1}{\e}\mathcal{M})^{\ell_0}\times [1-\delta_0,1+\delta_0]^{\ell_0}, j=1,\ldots, \ell_0}|\hat{P}_j(\gamma_0(p,s))-s_j|\le \frac{\delta_0}{4},
\end{aligned}
\end{equation}
\begin{equation}\label{nmi1}
\begin{aligned}
&\hat{\Upsilon}_j(\gamma_0(p,s))=p_j \mbox{ for } (p,s)\in \Big(\frac{1}{\e}\mathcal{M}\Big)^{\ell_0}\times [1-\delta_0,1+\delta_0]^{\ell_0} \mbox{ with } \min_{1\le i\neq j\le \ell_0}|p_i-p_j|\le \frac{kR_3}{2} \ \ \mbox{ and}\\
&\hat{P}_j(\gamma_0(p,s))=s_j \mbox{ for } (p,s)\in \Big(\frac{1}{\e}\mathcal{M}\Big)^{\ell_0}\times [1-\delta_0,1+\delta_0]^{\ell_0} \mbox{ with } \min_{1\le i\neq j\le \ell_0}|p_i-p_j|\le \frac{kR_3}{2}.
\end{aligned}
\end{equation}
We remark that there exists $R_4\ge kR_3$ such that
\begin{equation}\label{mo0}
\{\gamma_0(p,s)| (p,s)\notin A_\e(kR_3)\}\cap Z(\rho_4,R_4)=\emptyset.
\end{equation}
Indeed, choosing $R_4\ge kR_3$ large, for $\gamma_0(p,s)=\sum_{i=1}^{\ell_0}U_0(\cdot-p_i)+w$ satisfying $(p,s)\notin A_\e(kR_3)$ and $\|w\|_{H^1}< \frac{\rho_4}{2}$ and $u=\sum_{j=1}^{\ell_0} U_j(\cdot-y_j)+\bar{w}\in Z(\rho_4,R_4)$, we may assume that there exists $y_1\in \frac{1}{\e}N_\beta(\mathcal{M})$ such that $B(y_1,2R_0)\cap \{p_1,\cdots, p_{\ell_0}\}=\emptyset.$ Then, by \eqref{L},
\begin{align*}
\|\gamma_0(p,s)-u\|_{H^1}&=\|\sum_{i=1}^{\ell_0}U_0(\cdot-p_i)-\sum_{j=1}^{\ell_0} U_j(\cdot-y_j)+w-\bar{w}\|_{H^1}\\
&\ge\|U_1(\cdot-y_1)\|_{H^1(B(y_1,R_0))}-\sum_{i=1}^{\ell_0}\|U_0(\cdot-p_i)\|_{H^1(B(y_1,R_0))}\\
&-\sum_{j=2}^{\ell_0}\|U_j(\cdot-y_j)\|_{H^1(B(y_1,R_0))}-\frac{3}{2}\rho_4\\
&\ge \frac{3}{4}\rho_0-\frac{1}{4}\rho_0-\frac{3}{32}\rho_0>0.
\end{align*}
We set 
\[
 (\tilde{\Upsilon}_1(u),\cdots,\tilde{\Upsilon}_{\ell_0}(u))=\begin{cases}(\omega(\e\hat{\Upsilon}_1(u))/\e,\cdots, \omega(\e\hat{\Upsilon}_{\ell_0}(u)))/\e\\
\ \    \mbox{ if } u\in \{\gamma_0(p,s)\ |\ (p,s)\in \Big(\frac{1}{\e}\mathcal{M}\Big)^{\ell_0}\times [1-\delta_0,1+\delta_0]^{\ell_0}\},\\
 (\omega(\e \Upsilon_1(u))/\e,\cdots, \omega(\e \Upsilon_{\ell_0}(u)))/\e\\ 
\    \ \mbox{ if } u\in Z(\rho_4,R_4)\setminus \{\gamma_0(p,s)\ |\ (p,s)\in \Big(\frac{1}{\e}\mathcal{M}\Big)^{\ell_0}\times [1-\delta_0,1+\delta_0]^{\ell_0}\}\end{cases}
\]
 and 
\[
(\tilde{P}_1(u),\cdots,\tilde{P}_{\ell_0}(u))=\begin{cases}(H\circ \hat{P}_1(u),\cdots, H\circ \hat{P}_{\ell_0}(u))\\ \ \        \mbox{ if } u\in \{\gamma_0(p,s)\ |\ (p,s)\in \Big(\frac{1}{\e}\mathcal{M}\Big)^{\ell_0}\times [1-\delta_0,1+\delta_0]^{\ell_0}\},\\
( H\circ P\circ\hat{\tau}_{\e,1}(u),\cdots, H\circ P\circ\hat{\tau}_{\e,\ell_0}(u)) \\
\ \ \mbox{ if }  u\in Z(\rho_4,R_4)\setminus \{\gamma_0(p,s)  |  (p,s)\in \Big(\frac{1}{\e}\mathcal{M}\Big)^{\ell_0}\times [1-\delta_0,1+\delta_0]^{\ell_0}\},\end{cases}
\] where $\omega$ is given in \eqref{h3} and $H: \R\rightarrow[1- \delta_0,1+ \delta_0]$ is defined by
  \begin{equation}\label{mi1i}
 H(t)=\begin{cases}
                   1+\delta_0 & \mbox{if } t\ge 1+\frac34\delta_0, \\
                   1- \delta_0 & \mbox{if } t\le 1-\frac34\delta_0, \\
                   \frac43(t-1)+1 & \mbox{otherwise}.
                 \end{cases}
  \end{equation}

We remark that, for $\gamma_0(p,s)=\gamma_0(p_1,\cdots, p_{\ell_0},s_1,\cdots, s_{\ell_0})$, by defining $B_j=B(p_j,\frac52R_0)$ in \eqref{ji2} for $j=1,\cdots, \ell_0$,  we can define a map
\begin{align}\label{ni2}
(\tilde{\Upsilon}_1(\gamma_0(p,s)),&\ldots, \tilde{\Upsilon}_{\ell_0}(\gamma_0(p,s)), \tilde{P}_1(\gamma_0(p,s)), \ldots, \tilde{P}_{\ell_0}(\gamma_0(p,s))):\nonumber\\
&\Big(\frac{1}{\e}\mathcal{M}\Big)^{\ell_0}\times [1-\delta_0,1+\delta_0]^{\ell_0}\rightarrow \Big(\frac{1}{\e}\mathcal{M}\Big)^{\ell_0}\times [1-\delta_0,1+\delta_0]^{\ell_0}
\end{align}
 by $\tilde{\Upsilon}_i(\gamma_0(p,s))=\omega(\e\hat{\Upsilon}_i(\gamma_0(p,s)))/\e$ and $\tilde{P}_i(\gamma_0(p,s))=H\circ \hat{P}_i(\gamma_0(p,s))$, where $\hat{\Upsilon}_{i}$ and $\hat{\mathcal{P}}_{i}$ satisfy the properties \eqref{s1} and \eqref{nmi1} without a permutation. We see  that 
the map \eqref{ni2}  is homotopic to the identity map.  Indeed, we note first that, by \eqref{s1}, for all $j=1,\ldots, \ell_0$ and $0\le t\le 1$,
\[
\e |(1-t)\hat{\Upsilon}_{j}(\gamma_0(p,s))+tp_j-p_j|\le \e(1-t)3R_0.
\]
Then the homotopy is given by $\zeta:[0,1]\times\Big(\frac{1}{\e}\mathcal{M}\Big)^{\ell_0}\times [1-\delta_0,1+\delta_0]^{\ell_0}\rightarrow \Big(\frac{1}{\e}\mathcal{M}\Big)^{\ell_0}\times [1-\delta_0,1+\delta_0]^{\ell_0}$, where
\begin{equation}\label{mi1}
 \begin{aligned}
& \zeta(t,p,s)=\zeta(t,p_1,\cdots, p_{\ell_0},s_1,\cdots, s_{\ell_0})\\
&=\begin{cases}(\omega(\e((1-t)\hat{\Upsilon}_1(\gamma_0(p,s))+tp_1))/\e,\ldots,\omega(\e((1-t)\hat{\Upsilon}_{\ell_0}(\gamma_0(p,s))+tp_{\ell_0}))/\e,\\
\ \ (1-t)\tilde{P}_1(\gamma_0(p,s))+t s_1,\ldots,(1-t)\tilde{P}_{\ell_0}(\gamma_0(p,s))+t s_{\ell_0}).
\end{cases}
 \end{aligned}
\end{equation}
\subsection{An intersection property}
We give an intersection property, whch will be used to show an energy estimate from below $\ell_0 E_m$ for a deformation of the initial path $\gamma_0(p,s)$.

\begin{prop}\label{in0}
  For $r>R_4$, there exists $\bar{\e}_0 \in(0,\e_0)$ such that for $0<\e<\bar{\e}_0$ and any continuous map
  \begin{align*}
  h_\e : &[0,1]\times \Big(\frac{1}{\e}\mathcal{M}\Big)^{\ell_0}\times[1-\delta_0,1+\delta_0]^{\ell_0}\\
  &\rightarrow Z(\rho_4, R_4)\cup\{\gamma_0(p,s)\ |\ (p,s)\in \Big(\frac{1}{\e}\mathcal{M}\Big)^{\ell_0}\times[1-\delta_0, 1+\delta_0]^{\ell_0}\}
  \end{align*}
   satisfying
  \begin{equation}\label{in1}
    h_\e(0, p, s)=\gamma_0(p,s) \mbox{ for all } ( p, s)\in \Big(\frac{1}{\e}\mathcal{M}\Big)^{\ell_0}\times[1-\delta_0,1+\delta_0]^{\ell_0},
  \end{equation}
\begin{equation}\label{ji4}
h_\e(\tau+\sigma,p,s)=h_\e(\sigma,p,s) \mbox{ for all } \tau\ge 0 \mbox{ and all }(p,s)\in\Big(\frac{1}{\e}\mathcal{M}\Big)^{\ell_0}\times[1-\delta_0,1+\delta_0]^{\ell_0},
\end{equation}
  where $h_\e(\sigma,p,s) =\gamma_0(\bar{p},\bar{s})$ for some $\sigma\in[0,1]$ and for some  $(\bar{p},\bar{s})\in\{(p,s)\in \Big(\frac{1}{\e}\mathcal{M}\Big)^{\ell_0}\times[1-\delta_0,1+\delta_0]^{\ell_0}\ |\ |p_i-p_j|\le r  \mbox{ for some } 1\le i\neq j\le \ell_0 \mbox{ or } s_j\in\{1-\delta_0,1+\delta_0\} \mbox{ for some } j\}$, and 
\begin{equation}\label{ji1}
\Big\|\frac{d h_\e}{d\tau}(\tau,p,s)\Big\|_{H^1}\le 1 \mbox{ for all }(\tau,p,s)\in [0,1]\times \Big(\frac{1}{\e}\mathcal{M}\Big)^{\ell_0}\times[1-\delta_0,1+\delta_0]^{\ell_0},
\end{equation}
 there exists $(p^\e, s^\e)\in A_\e(r)$ such that $u_\e=h_\e(1, p^\e, s^\e)\in Z(\rho_4, R_4)$ satisfies
  \begin{equation}\label{in3}
    P(\hat{\tau}_{\e,j}(u_\e))=1 \mbox{ for } j=1, \ldots, \ell_0,
  \end{equation}
   \begin{equation}\label{in4}
    \frac{\beta_1}{\e}\le \Theta(u_\e)\le \frac{7\beta_1}{\e}.
  \end{equation}
\end{prop}
\begin{proof}
Choosing the order of balls in \eqref{ji2} along the path $\tau\mapsto h_\e(\tau,p,s)$, we can define a continuous map
\begin{equation}\label{mo3}
\Psi_\e: [0,1]\times\Big(\frac{1}{\e}\mathcal{M}\Big)^{\ell_0}\times[1-\delta_0,1+\delta_0]^{\ell_0}\rightarrow \Big(\frac{1}{\e}\mathcal{M}\Big)^{\ell_0}\times[1-\delta_0,1+\delta_0]^{\ell_0}
\end{equation}
by $\Psi_\e (\tau,p,s)=(\tilde{\Upsilon}_1( h_\e (\tau,p,s)),\ldots, \tilde{\Upsilon}_{\ell_0}( h_\e (\tau,p,s)), \tilde{P}_1( h_\e (\tau,p,s)), \ldots, \tilde{P}_{\ell_0}( h_\e (\tau,p,s))),$ where $\Psi_\e( 0,p,s)=(\tilde{\Upsilon}_1(\gamma_0(p,s)),\ldots, \tilde{\Upsilon}_{\ell_0}(\gamma_0(p,s)), \tilde{P}_1(\gamma_0(p,s)), \ldots, \tilde{P}_{\ell_0}(\gamma_0(p,s)))$ is defined in \eqref{ni2}. Indeed, for $h_\e (\tau,p,s)   \in Z(\rho_4,R_4)\cup \{\gamma_0(p,s)\ |\ (p,s)\in A_\e(r)\}$, we denote 
\[
h_\e (\tau,p,s)=\sum_{j=1}^{\ell_0} U_j^\tau(\cdot-y_j^\tau)+w^\tau,
\] where $y_j^\tau\in \frac{1}{\e}N_\beta(\mathcal{M})$, $U_j^\tau\in S_m, j=1,\cdots, \ell_0$, $|y_j^\tau-y_{j^\prime}^\tau|\ge R_4>36R_0$ for $1\le j \neq j^\prime\le \ell_0$ and $\|w^\tau\|_{H^1}<\rho_4$. If $h_\e (\tau,p,s)\in  Z(\rho_4,R_4)\cup \{\gamma_0(p,s)\ |\ (p,s)\in A_\e(r)\}$ for $\tau\in[\tau_1,\tau_2]$, where $\tau_1, \tau_2\in [0,1]$ and $|\tau_1-\tau_2|<\frac{\rho_0}{8}$, then we see that for a permutation $\sigma$ of $(1, 2,\cdots, \ell_0)$, 
\begin{equation}\label{mo1}
|y_i^{\tau_1}-y_{\sigma(i)}^{\tau_2}|<2R_0 \mbox{ for all } i=1,\cdots, \ell_0.
\end{equation} Contrary to our claim \eqref{mo1}, suppose that there exists $y_1^{\tau_1}\in \frac{1}{\e}N_\beta(\mathcal{M})$ such that $|y_1^{\tau_1}-y_j^{\tau_2}|\ge 2R_0$ for all $j=1, \cdots, \ell_0$. Then, by \eqref{L}, \eqref{llo1} and \eqref{ji1}, we have
\begin{align*}
|\tau_1-\tau_2| &\ge\|h_\e (\tau_1,p,s)-h_\e (\tau_2,p,s)\|_{H^1}\\
&\ge\|\sum_{j=1}^{\ell_0} U_j^{\tau_1}(\cdot-y_j^{\tau_1})+w^{\tau_1} -\Big(\sum_{j=1}^{\ell_0} U_j^{\tau_2}(\cdot-y_j^{\tau_2})+w^{\tau_2}\Big)\|_{H^1(B(y_1^{\tau_1},R_0))}\\
&\ge \|  U_1^{\tau_1}(\cdot-y_j^{\tau_1})\|_{H^1(B(y_1^{\tau_1},R_0))}-\sum_{j=2}^{\ell_0}\| U_j^{\tau_1}(\cdot-y_j^{\tau_1})\|_{H^1(B(y_1^{\tau_1},R_0))}\\
&-\sum_{j=1}^{\ell_0}\| U_j^{\tau_2}(\cdot-y_j^{\tau_2})\|_{H^1(B(y_1^{\tau_1},R_0))}-\|w^{\tau_1}\|_{H^1}-\|w^{\tau_2}\|_{H^1}\\
&>\frac34 \rho_0-\frac18 \rho_0-\frac18 \rho_0-\frac{1}{16}\rho_0-\frac{1}{16}\rho_0=\frac38 \rho_0,
\end{align*}
which is a contradiction.
On the other hand, by \eqref{in1} and \eqref{ji4}, if $h_\e (\bar{\tau},p,s)=\gamma_0(\bar{p},\bar{s})$ for some $\bar{\tau}\in[0,1]$, $(p,s)\in A_\e(r)$ and $(\bar{p},\bar{s})\in \Big(\frac{1}{\e}\mathcal{M}\Big)^{\ell_0}\times[1-\delta_0,1+\delta_0]^{\ell_0}$ satisfying $\xi(\bar{p})=\min_{1\le i\neq j\le\ell_0}|\bar{p}_i-\bar{p}_j|=r$ or $\bar{s}\in \partial([1-\delta_0,1+\delta_0]^{\ell_0})$, then 
\begin{equation}\label{mo2}
h_\e ({\tau},p,s)=\gamma_0(\bar{p},\bar{s}) \mbox{ for all }\tau\ge \bar{\tau}.
\end{equation}
Therefore, by \eqref{mo0}, \eqref{mo1} and \eqref{mo2}, choosing the order of balls in \eqref{ji2} along the path $\tau\mapsto h_\e(\tau,p,s)$, we can prove the claim \eqref{mo3}.

  By \eqref{pp2}, we  fix
\begin{align*}
&q=(q_1,\ldots, q_{\ell_0})\in \begin{cases}
                            (c)^{\ell_0}  & \mbox{if } (M1)\mbox{ holds}, \\                   
                            (\mathcal{M})^{\ell_0} & \mbox{if } (M2) \mbox{ holds},
                          \end{cases}
\end{align*}
with $\xi(q_1,\ldots, q_{\ell_0})=\min_{1\le i\neq j\le \ell_0}|q_i-q_j|=4\beta_1$, where the set $c$ is defined in \eqref{h4}, and denote $Q_\e=(q_1/\e, \ldots, q_{\ell_0}/\e,1, \ldots, 1)\in \Big(\frac{1}{\e}\mathcal{M}\Big)^{\ell_0}\times[1-\delta_0,1+\delta_0]^{\ell_0}$.
   By \eqref{in1}, \eqref{mo3} and the fact that the map \eqref{ni2}  is homotopic to the identity map, we have the map
   \begin{equation}\label{o4}
      \Psi_\e(0, p, s)=(\tilde{\Upsilon}_1(\gamma_0(p,s)),\ldots, \tilde{\Upsilon}_{\ell_0}(\gamma_0(p,s)), \tilde{P}_1(\gamma_0(p,s)), \ldots, \tilde{P}_{\ell_0}(\gamma_0(p,s))),
   \end{equation}
which is homotopic to the identity map.
By \eqref{s1}, \eqref{mi1i}, \eqref{in1} and \eqref{ji4}, if $(p,s)\in \Big(\frac{1}{\e}\mathcal{M}\Big)^{\ell_0}\times \partial([1-\delta_0,1+\delta_0]^{\ell_0})$, then for $\tau\in[0, 1]$
\begin{equation}\label{ccc1}
\Psi_\e(\tau,p,s)\in \Big(\frac{1}{\e}\mathcal{M}\Big)^{\ell_0}\times \partial([1-\delta_0,1+\delta_0]^{\ell_0})
\end{equation}
and for $t\in [0,1]$
\begin{equation}\label{ccc0}
\zeta(t,p,s)\in \Big(\frac{1}{\e}\mathcal{M}\Big)^{\ell_0}\times \partial([1-\delta_0,1+\delta_0]^{\ell_0}),
\end{equation}
where $\zeta$ is given in \eqref{mi1}.
Since $S^{\ell_0}\cong I^{\ell_0}\cup_i I^{\ell_0}$, where $I=[1-\delta_0,1+\delta_0]$,  $i:\partial I^{\ell_0}\rightarrow I^{\ell_0}$ is the inclusion and the definition of $I^{\ell_0}\cup_i I^{\ell_0}$ is given in  Definition \ref{me1},
by \eqref{o4}-\eqref{ccc0}, we can extend the map $\Psi_\e(1,\cdot,\cdot):\Big(\frac{1}{\e}\mathcal{M}\Big)^{\ell_0}\times [1-\delta_0,1+\delta_0]^{\ell_0}\rightarrow \Big(\frac{1}{\e}\mathcal{M}\Big)^{\ell_0}\times [1-\delta_0,1+\delta_0]^{\ell_0}$ to a map $\tilde{f}: \Big(\frac{1}{\e}\mathcal{M}\Big)^{\ell_0}\times S^{\ell_0}\rightarrow \Big(\frac{1}{\e}\mathcal{M}\Big)^{\ell_0}\times S^{\ell_0},$ where $\tilde{f}$ is homotopic to the identity map.
We note that
\begin{align*}
\mathcal{M}^{\ell_0}\times S^{\ell_0} &\mbox{ is a connected compact }(d+1)\ell_0\mbox{-dimensional  }\\
&\ \ \ \mbox{ embedded topological submanifold  of }\R^{\ell_0(N+1)+1}
\end{align*}
if $\mathcal{M}$ is a connected compact $d$-dimensional embedded topological  submanifold of  $\R^N$ for some $d=1,\cdots, N-1$. Moreover, by \cite[\rom{6} Theorem 1.6]{Bredon}
\begin{align*}
\mathcal{M}^{\ell_0}&\times S^{\ell_0} \mbox{ is a finite }(d+1)\ell_0\mbox{-dimensional  polyhedron }  \\
&\mbox{ with }0\neq[\bar{c}]\times \cdots \times[\bar{c}]\times [S^{\ell_0}]\in\Big(H_\ast(\mathcal{M}^{\ell_0};\Z_2)\otimes H_\ast(S^{\ell_0};\Z_2)\Big)_{(d+1)\ell_0} \\
&\ \ \ \ \ \ \ \ \  \ \ \ \ \ \ \ \  \ \ \ \ \ \ \ \ \ \ \ \ \ \ \ \ \ \ \ \ \ \  = H_{(d+1)\ell_0}\Big(\mathcal{M}^{\ell_0}\times S^{\ell_0};\Z_2\Big)
\end{align*}
if $\mathcal{M}$ is a finite $d$-dimensional polyhedron with $0\neq \bar{c}\in H_d(\mathcal{M};\Z_2)$ for some $ d=1,\ldots, N-1,$  where $\bar{c}$ is given in \eqref{h4} above.
     By  Lemma \ref{y2} and Lemma \ref{y1}, there exists $(p^\e,s^\e)\in \Big(\frac{1}{\e}\mathcal{M}\Big)^{\ell_0}\times [1-\delta_0,1+\delta_0]^{\ell_0}$ such that $\Psi_\e(1,p^\e,s^\e)=Q_\e$. Then $u_\e=h_\e(1,p^\e,s^\e)$ satisfies  for $j=1,\cdots,\ell_0$,
  \begin{align*}
    \tilde{\Upsilon}_j(u_\e)=q_j/\e \mbox{ and } \tilde{P}_j(u_\e)=1.
  \end{align*}
If $u_\e=\gamma_0(\bar{p},\bar{s})$, where $(\bar{p},\bar{s})\in \Big(\frac{1}{\e}\mathcal{M}\Big)^{\ell_0}\times [1-\delta_0,1+\delta_0]^{\ell_0}$ satisfying $\xi(\bar{p})=\min_{1\le i\neq j\le\ell_0}|\bar{p}_i-\bar{p}_j| =|\bar{p}_{i_0}-\bar{p}_{j_0}|\le r$, then, by  \eqref{pil} and \eqref{s1}, we have for a permutation $\sigma$ of $(1,\cdots,\ell_0)$, 
\begin{align*}
|\omega(\e\hat{\Upsilon}_{\sigma(j)}(\gamma_0(\bar{p},\bar{s})))-\omega(\e \bar{p}_j)|&\le |\omega(\e\hat{\Upsilon}_{\sigma(j)}(\gamma_0(\bar{p},\bar{s})))-\e\hat{\Upsilon}_{\sigma(j)}(\gamma_0(\bar{p},\bar{s}))|+|\e\hat{\Upsilon}_{\sigma(j)}(\gamma_0(\bar{p},\bar{s}))-\epsilon\bar{p}_j|\\
&<\frac{3}{2}\beta_1\  \mbox{ for }\ j=1,\cdots,\ell_0.
\end{align*}  Then, by  \eqref{pp2} , we have for a permutation $\sigma$ of $(1,2,\cdots,\ell_0)$,
\begin{align*}
\frac{4\beta_1}{\e}&= \frac{1}{\e}\min_{1\le i\neq j\le \ell_0}{|q_{i}-q_{j}|}\\
&\le\frac{|\omega(\e\hat{\Upsilon}_{\sigma(i_0)}(\gamma_0(\bar{p},\bar{s})))-\e\hat{\Upsilon}_{\sigma(j_0)}(\gamma_0(\bar{p},\bar{s}))|}{\e}\\
&\le \frac{|\omega(\e\hat{\Upsilon}_{\sigma(i_0)}(\gamma_0(\bar{p},\bar{s})))-\omega(\e \bar{p}_{i_0})|}{\e}+\frac{|\omega(\e \bar{p}_{i_0})-\omega(\e \bar{p}_{j_0})|}{\e}+\frac{|\omega(\e\hat{\Upsilon}_{\sigma(j_0)}(\gamma_0(\bar{p},\bar{s})))-\omega(\e \bar{p}_{j_0})|}{\e}\\
&<\frac{3\beta_1}{\e}+\frac{|\omega(\e \bar{p}_{i_0})-\omega(\e \bar{p}_{j_0})|}{\e}\le\frac{3\beta_1}{\e}+r,
\end{align*}
which is a contradiction.
We deduce that $u_\e \in Z(\rho_4, R_4)$, $(p^\e,s^\e)\in A_\e(r)$ and $ P(\hat{\tau}_{\e,j}(u_\e))=1 \mbox{ for } j=1, \ldots, \ell_0$. Moreover, we see that, by \eqref{pp2} and \eqref{pil},
\begin{align*}
  \Theta(u_\e)&=\min_{1\le i\neq j\le \ell_0}|\Upsilon_i(u_\e)-\Upsilon_j(u_\e)|\nonumber\\
              &\ge \min_{1\le i\neq j\le \ell_0}(|\tilde{\Upsilon}_i(u_\e)-\tilde{\Upsilon}_j(u_\e)|-|\tilde{\Upsilon}_i(u_\e)-\Upsilon_i(u_\e)|-|\tilde{\Upsilon}_j(u_\e)-\Upsilon_j(u_\e)|)\\
&\ge \frac{4\beta_1}{\e}-\max_{i=1,\cdots ,\ell_0}\Big|\frac{\omega(\e\Upsilon_i(u_\e))}{\e}-\Upsilon_i(u_\e)\Big|-\max_{j=1,\cdots, \ell_0}\Big|\frac{\omega(\e\Upsilon_j(u_\e))}{\e}-\Upsilon_j(u_\e)\Big|\\
              &\ge \frac{4\beta_1}{\e}-\frac{3\beta_1}{\e}=\frac{\beta_1}{\e} 
\end{align*}
and similarly, by \eqref{pp2} and \eqref{pil}, $\Theta(u_\e)\le \frac{7\beta_1}{\e}$. This completes the proof of Proposition \ref{in0}.
\end{proof}

\begin{cor}\label{lw}
  For $r> R_4$, $h_\e(\tau,p,s)\in C([0,1]\times \Big(\frac{1}{\e}\mathcal{M}\Big)^{\ell_0}\times[1-\delta_0,1+\delta_0]^{\ell_0},Z(\rho_4, R_4)\cup\{\gamma_0(p,s)\ |\ (p,s)\in \Big(\frac{1}{\e}\mathcal{M}\Big)^{\ell_0}\times[1-\delta_0, 1+\delta_0]^{\ell_0}\})$ satisfies \eqref{in1}-\eqref{ji1}. Then the following lower estimate holds
  \begin{equation}\label{in7}
  \liminf_{\e\rightarrow0} \max_{(p,s)\in A_\e(r)}I_\e(h_\e(1,p,s))\ge \ell_0 E_m
  \end{equation}
\end{cor}
\begin{proof}
  Assume $h_\e\in C([0,1]\times \Big(\frac{1}{\e}\mathcal{M}\Big)^{\ell_0}\times[1-\delta_0,1+\delta_0]^{\ell_0}, Z(\rho_4, R_4)\cup\{\gamma_0(p,s)\ |\ (p,s)\in \Big(\frac{1}{\e}\mathcal{M}\Big)^{\ell_0}\times[1-\delta_0, 1+\delta_0]^{\ell_0}\})$ satisfies \eqref{in1}-\eqref{ji1}. By Proposition \ref{in0}, there exists $(p^\e, s^\e)\in A_\e(r)$ such that  $u_\e=h_\e(1,p^\e,s^\e)\in Z(\rho_4, R_4)$ satisfying \eqref{in3} and \eqref{in4}. By (i) of Lemma \ref{T6}, we have $I_\e(u_\e)\ge I_\e(\tau_\e(u_\e))$. By (ii) of Lemma \ref{T6} and \eqref{in4},
  \[
   \Big\|\sum_{j=1}^{\ell_0}\hat{\tau}_{\e,j}(u_\e)-\tau_\e(u_\e)\Big\|_{H^1}\le A_4\exp(-A_5\Theta(u_\e))\rightarrow0 \mbox{ as } \e\rightarrow0.
   \]
   Then as $\e\rightarrow0$,
   \begin{equation}\label{in5}
   I_\e(u_\e)\ge I_\e(\sum_{j=1}^{\ell_0}\hat{\tau}_{\e,j}(u_\e))+o(1)=\sum_{j=1}^{\ell_0}I_\e(\hat{\tau}_{\e,j}(u_\e))+o(1).
   \end{equation}
   By \eqref{p1} and \eqref{in3}, we have
   \begin{equation}\label{in6}
     L_m(\hat{\tau}_{\e,j}(u_\e))\ge E_m.
   \end{equation}
   It follows from (V2), \eqref{in4}, \eqref{in5} and \eqref{in6} that as $\e \rightarrow
   0$,
   \begin{align*}
     I_\e(u_\e)&\ge\sum_{j=1}^{\ell_0}I_\e(\hat{\tau}_{\e,j}(u_\e))+o(1)\\
              &=\sum_{j=1}^{\ell_0}L_m(\hat{\tau}_{\e,j}(u_\e))+\frac{1}{2}\sum_{j=1}^{\ell_0}\int_{B(\Upsilon_j(u_\e),\frac{\Theta(u_\e)}{3}+1)}(V(\e x)-m)(\hat{\tau}_{\e,j}(u_\e))^2dx+o(1)\\
              &\ge \ell_0 E_m+o(1).
   \end{align*}
Thus, we get \eqref{in7}.
\end{proof}
\section{Deformation argument}\label{Deformation argument}
The aim of this section is to complete the proof of our Theorem \ref{main thm}. First we will show via a deformation argument that for any given $\e_1>0, I_\e$ has a critical point $u$ in $Z(\rho_4,R_4)$ with a critical value in $[\ell_0E_m-\e_1, b_\e]$ for small $\e$, where $b_\e$ is a sequence satisfying $b_\e \rightarrow \ell_0E_m$ as $\e\rightarrow 0$.

From now on, for $A, B\subset H^1(\R^N)$, we denote
\[
\operatorname{dist}_{H^1}(A,B)=\inf\{\|a-b\|_{H^1} \ | \ a\in A, b\in B\}.
\]
We can see that there exists a constant $d_0>0$ independent of $\e$ such that
\[
\operatorname{dist}_{H^1}(Z(\frac{1}{2}\rho_4,2R_4),\partial Z(\rho_4, R_4))\ge 2d_0 \mbox{ for small } \e.
\]
Without loss of generality, we may assume
\begin{equation}\label{e0}
\e_1<\frac{1}{2}\nu_0d_0,
\end{equation}
where $\nu_0$ is given in Proposition \ref{g1}.
We take and fix $r>R_4$ such that for small $\e$
\begin{equation}\label{d1}
  \{\gamma_0(p,s)\ | \ (p,s)\in A_\e(r)\}\subset Z(\frac{\rho_4}{2}, 2R_4)
\end{equation}
and set
\begin{equation}\label{d0}
b_\e=\max_{(p,s)\in A_\e(r)}I_\e(\gamma_0(p,s)).
\end{equation}
Such a $r$ exists by \eqref{i4}. By Lemma \ref{i2} and Proposition \ref{inter}, we have  for small $\e$,
\begin{equation}\label{ju1}
b_\e\le \ell_0E_m+\e_1
\end{equation}
and
\begin{equation}\label{d2}
  \max_{(p,s)\in \partial A_\e(r)} I_\e(\gamma_0(p,s))\le \ell_0 E_m-\nu_1.
\end{equation}
By Proposition \ref{g1}, there exists $\nu_0>0$ such that for small $\epsilon>0$
\begin{equation}\label{d3}
\|I_\epsilon^\prime(u)\|_{H^{-1}}\ge \nu_0
\end{equation}
for all $u\in Z(\rho_4,R_4)\setminus Z(\frac{1}{2}\rho_4,2R_4)$ with $I_\epsilon(u)\le b_\epsilon$.
In what follows, we fix such $r$ and consider small $\e$ for which \eqref{d1}-\eqref{d3} hold.

We argue indirectly and  suppose that $I_\e^\prime(u)\neq 0$ for $u\in Z(\rho_4, R_4)$ with $I_\e(u)\in [\ell_0E_m-\e_1, b_\e].$ By Proposition \ref{g2}, we can find $\kappa_\e>0$ such that
\[
\|I_\e^\prime(u)\|_{H^{-1}}\ge \kappa_\e \mbox{ for } u\in Z(\rho_4, R_4)\mbox{ with } I_\e(u)\in [\ell_0E_m-\e_1, b_\e].
\]
We may assume $\kappa_\e<\nu_0$.

For $u\in Z(\rho_4, R_4)$ with $I_\e(u)\le b_\e$, we consider the following ODE:
\begin{equation}\label{d4}
  \begin{cases}
    &\frac{d\eta}{d\tau}=-\varphi_1(I_\e(\eta))\varphi_2(\operatorname{dist}_{H^1}(\eta, Z(\frac{\rho_4}{2}, 2R_4)))\frac{I_\e^\prime(\eta)}{\|I_\e^\prime(\eta)\|_{H^{-1}}},\\
    &\eta(0,u)=u,
  \end{cases}
\end{equation}
where $\varphi_1(\xi), \varphi_2(\xi): \R\rightarrow [0,1]$ are Lipschitz continuous functions such that
\begin{align*}
  &\varphi_1(\xi)=\begin{cases}
                   1 & \mbox{if } \xi\ge \ell_0E_m-\frac{1}{2}\min\{\nu_1,\e_1\}, \\
                   0 & \mbox{if } \xi\le \ell_0E_m-\min\{\nu_1, \e_1\},
                 \end{cases}\\
  &\varphi_2(\xi)=\begin{cases}
                   1 & \mbox{if } \xi\le d_0, \\
                   0 & \mbox{if } \xi\ge 2d_0.
                 \end{cases}
\end{align*}
For any $u\in Z(\rho_4,R_4)\cap \{ u \ | \ I_\e(u)\le b_\e\}$, \eqref{d4} has a unique solution $\eta(\tau, u)$ at least locally in time $\tau$ and since
\begin{align*}
\frac{d}{d\tau}I_\e(\eta(\tau,u))&=I_\e^\prime(\eta(\tau,u))\frac{d\eta}{d\tau}\\
                                 &=-\varphi_1(I_\e(\eta))\varphi_2(\operatorname{dist}_{H^1}(\eta, Z(\frac{\rho_4}{2}, 2R_4)))\|I_\e^\prime(\eta)\|_{H^{-1}},
\end{align*}
we can see the following properties in a standard way
\begin{enumerate}[(a)]
  \item $\eta(0,u)=u$ for all $u$ with $I_\e(u)\le b_\e$;
  \item $\eta(\tau, u)=u$ for all $\tau\ge 0$ if $I_\e(u)\le \ell_0E_m-\min\{\nu_1,\e_1\};$
  \item $\|\frac{d\eta}{d\tau}(\tau, u)\|_{H^{1}}\le 1$ for all $\tau$ and $u$ with $I_\e(u)\le b_\e$;
  \item $\frac{d}{d\tau}I_\e(\eta(\tau, u))\le 0$ for all $\tau$ and $u$ with $I_\e(u)\le b_\e$;
  \item $\frac{d}{d\tau}I_\e(\eta(\tau, u))\le -\nu_0$ if $\eta\in (Z(\rho_4,R_4)\setminus Z(\frac{\rho_4}{2},2R_4))\cap \{u\ | \ I_\e(u)\in [\ell_0E_m-\frac{1}{2}\min\{\nu_1,\e_1\}, b_\e]\}$ and $\operatorname{dist}_{H^1}(\eta,Z(\frac{\rho_4}{2},2R_4))\le d_0$;
  \item $\frac{d}{d\tau}I_\e(\eta(\tau, u))\le -\kappa_\e$ if $\eta\in Z(\rho_4,R_4)\cap \{u\ | \ I_\e(u)\in [\ell_0E_m-\frac{1}{2}\min\{\nu_1,\e_1\}, b_\e]\}$\\
 and $\operatorname{dist}_{H^1}(\eta,Z(\frac{\rho_4}{2},2R_4))\le d_0$.
\end{enumerate}
Now we claim the following lemma.
\begin{lem}\label{d5}
  Let
  \begin{equation}\label{e1}
  T_\e=\frac{1}{\kappa_\e}(\frac{1}{2}\min\{\nu_1,\e_1\}+\e_1).
  \end{equation}
  Then the followings hold
  \begin{enumerate}[(i)]
    \item for any $(p,s)\in A_\e(r)$, the solution $\eta(\tau, \gamma_0(p,s))$ exists globally in time $\tau\ge0$ and satisfies $\eta(\tau,\gamma_0(p,s))\in Z(\rho_4,R_4)$ for all $\tau\ge 0$;
    \item for $(p,s)\in \partial A_\e(r), \eta(\tau, \gamma_0(p,s))=\gamma_0(p,s)$ for all $\tau \ge 0$;
    \item $\eta(\tau+\sigma,\gamma_0(p,s))=\eta(\tau,\eta(\sigma,\gamma_0(p,s)))$ for all $\tau, \sigma\ge 0$ and $(p,s)\in A_\e(r)$;
    \item $I_\e(\eta(T_\e,\gamma_0(p,s)))\le \ell_0 E_m-\frac{1}{2}\min\{\nu_1,\e_1\}$ for all $(p,s)\in A_\e(r)$.
  \end{enumerate}
\end{lem}
\begin{proof}
   $(i)$ The solution $\eta(\tau,u)$ of \eqref{d4} is extendable as long as $\eta(\tau,u)\in Z(\rho_4,R_4)$ and $I_\e(\eta)\in [\ell_0E_m-\e_1, b_\e]$. We note that the right hand side of the first equation \eqref{d4} vanishes on $\partial Z(\rho_4,R_4)$. Then, by the property (b) and the property (d),  we deduce that for $u\in Z(\rho_4,R_4)$ with $I_\e(u)\le b_\e$, $\eta(\tau,u)$ exists globally and $\eta(\tau,u)\in Z(\rho_4,R_4)$ . In particular, by \eqref{d1} and \eqref{d0}, for $u=\gamma_0(p,s)$ with $(p,s)\in A_\e(r)$, we see that $\eta(\tau,\gamma_0(p,s))$ exists globally in $\tau\ge 0$ and $\eta(\tau,\gamma_0(p,s))\in Z(\rho_4,R_4)$ for all $\tau\ge 0$.\\
   $(ii)$ The proof follows directly from property (b) and \eqref{d2}.\\
   $(iii)$ The semigroup property for solutions of \eqref{d4} gives $(iii)$.\\
  $(iv)$ Arguing indirectly, we suppose that  for $u=\gamma_0(p,s)$,
  \[
  I_\e(\eta(\tau,u))> \ell_0 E_m-\frac{1}{2}\min\{\nu_1,\e_1\} \mbox{ for } \tau\in [0, T_\e].
  \]
  We consider two cases:
  \begin{enumerate}[(A)]
    \item there exists $\tau_1\in [0,T_\e]$ such that $\operatorname{dist}_{H^1}(\eta(\tau_1,u),Z(\frac{1}{2}\rho_4,2R_4))=d_0$;
    \item $\operatorname{dist}_{H^1}(\eta(\tau,u),Z(\frac{1}{2}\rho_4,2R_4))<d_0$ for all $\tau\in [0,T_\e].$
  \end{enumerate}
  Suppose that (A) takes a place. Since $u=\gamma_0(p,s)\in Z(\frac{\rho_4}{2},2R_4)$, we may assume that $\tau_1$ is the first time that satisfies dist$_{H^1}(\eta(\tau_1,u),Z(\frac{\rho_4}{2},2R_4))=d_0$. By the property (c), we can see $\tau_1\ge d_0$ and
  \[
  \eta(\tau,u)\in Z(\rho_4,R_4)\setminus Z(\frac{\rho_4}{2},2R_4) \mbox{ for } \tau\in [\tau_1-d_0,\tau_1].
  \]
  Thus by the property (d), the property (e), \eqref{e0} and \eqref{ju1}, we have
  \begin{align*}
  &\ell_0 E_m-\frac{1}{2}\min\{\nu_1,\e_1\}<I_\e(\eta(T_\e,u))=I_\e(u)+\int_{0}^{T_\e}\frac{d}{d\tau}I_\e(\eta(\tau,u))d\tau\\
  &\le I_\e(u)+\int_{\tau_1-d_0}^{\tau_1}\frac{d}{d\tau}I_\e(\eta(\tau,u))d\tau\le I_\e(u)-\int_{\tau_1-d_0}^{\tau_1}\nu_0d\tau\\
  &\le (\ell_0E_m+\e_1)-\nu_0d_0\le \ell_0E_m-\e_1.
  \end{align*}
This is a contradiction and (A) cannot take a place.  Next suppose that (B) occurs. Then by the property (f) and \eqref{ju1},
  \begin{align*}
    &\ell_0 E_m-\frac{1}{2}\min\{\nu_1,\e_1\}<I_\e(\eta(T_\e,u))=I_\e(u)+\int_{0}^{T_\e}\frac{d}{d\tau}I_\e(\eta(\tau,u))d\tau\\
    &\le I_\e(u)-\int_{0}^{T_\e}\kappa_\e d\tau\le \ell_0E_m+\e_1-\kappa_\e T_\e,
  \end{align*}
  which is a contradiction to the definition \eqref{e1} of $T_\e$. Thus $(iv)$ holds.
\end{proof}
  \noindent{\bf Completion of the proof of Theorem \ref{main thm}.}
  By Lemma \ref{d5} $(ii)$, we can extend the deformation $\eta(\cdot,\gamma_0(\cdot,\cdot))$ on $[0,T_\e]\times A_\e(r)$ to the function on $[0,T_\e]\times\Big(\frac{1}{\e}\mathcal{M}\Big)^{\ell_0}\times[1-\delta_0,1+\delta_0]^{\ell_0}$(still denote $\eta(\cdot,\gamma_0(\cdot,\cdot))$) so that
  \[
  \eta(\tau,\gamma_0(p,s))=\gamma_0(p,s) \mbox{ for } (\tau,p,s)\in [0,T_\e]\times \Big(\Big(\Big(\frac{1}{\e}\mathcal{M}\Big)^{\ell_0}\times[1-\delta_0,1+\delta_0]^{\ell_0}\Big)\setminus A_\e(r)\Big).
  \]
  By Lemma \ref{d5} $(iv)$, for small $\e>0$, it holds that
  \[
  I_\e(\eta(T_\e,\gamma_0(p,s)))\le \ell_0 E_m-\frac{1}{2}\min\{\nu_1,\e_1\}\ \mbox{ for all }\ (p,s)\in A_\e(r).
  \]
  Applying Corollary \ref{lw} to $\eta(\tau, \gamma_0(p,s))$, we have
  \[
  \liminf_{\e\rightarrow0} \max_{(p,s)\in A_\e(r)}I_\e(\eta(T_\e,\gamma_0(p,s)))\ge \ell_0 E_m,
  \]
  which is a contradiction. Thus for small $\e$, $I_\e$ has a critical point in $Z(\rho_4, R_4)\cap\{u\ | \ I_\e(u)\in [\ell_0E_m-\e_1, b_\e]\}.$

  Since we can choose $\e_1$ arbitrarily small, we can see \eqref{main eq} has a solution $u_\e$ for small $\e$ such that
  \begin{equation}\label{d6}
  \begin{aligned}
    &u_\e\in Z(\rho_4,R_4)\\
    &I_\e(u_\e)\rightarrow \ell_0 E_m \mbox{ as } \e\rightarrow0.
  \end{aligned}
  \end{equation}
  By the choice of $\rho_4$, especially $\rho_4<\rho_0/8$, and \eqref{d6}, from the proof of Proposition \ref{g1}, after extracting a subsequence, still denoted by $\e$, there exist sequences $(x_\e^k)_{k=1}^{\ell_0}, (W^k)_{k=1}^{\ell_0}\subset S_m$ satisfying
  \begin{align*}
    &\lim_{\e\rightarrow0}\operatorname{dist}(x_\e^k,\mathcal{M})=0 \mbox{ for all } k\in \{1,\ldots, \ell_0\},\\
    &\lim_{\e\rightarrow0}\min_{1\le i\neq j\le \ell_0}|x_\e^i/\e-x_\e^j/\e|=\infty,\\
    &\|u_\e-\sum_{k=1}^{\ell_0}W^k(\cdot-x_\e^k/\e)\|_{H^1}\rightarrow 0 \mbox{ as } \e \rightarrow0.
  \end{align*}
  This completes the proof of Theorem \ref{main thm}. $\Box$

\section*{Acknowledgement}
The paper is a part of the author’s thesis. The author is grateful to his thesis adviser Professor Jaeyoung Byeon for suggesting this type of problem and for all his encouragements and useful discussions.

\section{Appendix}
We shall need some topological tools that we now present for the reader convenience. We   follow notations in \cite{Bredon}.

Points $v_0,\cdots, v_n\in \R^\infty$ are ``affinely independent'' if they span an affine $n$-plane, i.e. if 
\[
\Big(\sum_{i=0}^n\lambda_iv_i=0, \sum_{i=0}^n\lambda_i=0\Big)	\Rightarrow \mbox{ all } \lambda_i=0.
\]
If $v_0,\cdots, v_n\in \R^\infty$ are  affinely independent then we put
\[
\sigma=(v_0,\cdots, v_n)=\Big\{\sum_{i=0}^n\lambda_iv_i\ | \ \sum_{i=0}^n\lambda_i=1, \lambda_i\ge 0\Big\}, 
\]
which is the ``affine simplex" spanned by the $v_i$. For $k\le n$, a $k$-face of $\sigma$ is any affine simplex of the form $(v_{i_0},\cdots, v_{i_k})$ where these vertices are all distinct and so are affinely independent.
\begin{defn}
A (geometric) simplicial complex  $\mathcal{S}$ is a collection of affine simplices such that
\begin{enumerate}[(i)]
\item if $\sigma \in  \mathcal{S}$, then any face of $\sigma$ is in  $\mathcal{S}$;
\item if $\sigma, \tau\in  \mathcal{S}$, then $\sigma\cap \tau$ is a face of both $\sigma$ and $\tau$, or is empty;
\end{enumerate}
\end{defn}
A $s$-dimensional simplicial complex (or simplicial $s$-complex)  $\mathcal{S}$ is a simplicial complex where the largest dimension of any simplex in $\mathcal{S}$ equals $s$. Let $K$ be a finite simplicial complex and choose an ordering of the vertices $v_0, v_1, \ldots$ of $K$ and $K^{[r]}$ denotes the r-fold iterated barycentric subdivision of $K$. Define mesh$(K)$ to be the maximum diameter of a simplex of $K$. We put $|K|=\bigcup\{\sigma|\sigma\in K\}$. This is called the polyhedron of $K$. Also $C_k(K;\Z_2)$ denotes  the free abelian group generated by the $k$-simplices $\langle v_{\sigma_0}, \ldots, v_{\sigma_k}\rangle$ of $K$ with coefficients in $\Z_2$, where $\sigma_0<\cdots< \sigma_k$  and $\partial_k: C_k(K;\Z_2)\rightarrow C_{k-1}(K;\Z_2)$ denotes a homomorphism  given by
\begin{equation}\label{ap}
\partial_k \langle v_{\sigma_0}, \ldots, v_{\sigma_k}\rangle = \sum_{i=0}^{k} (-1)^i \langle v_{\sigma_0}, \ldots, \hat{v}_{\sigma_i},\ldots, v_{\sigma_k}\rangle,
\end{equation}
where $\sigma=\langle v_{\sigma_0}, \ldots, v_{\sigma_k}\rangle$ is  a simplex of $K$ satisfying $\sigma_0<\cdots< \sigma_k$ and the `hat' symbol $\ \hat{}\ $ over $v_{\sigma_i}$ indicates that this vertex is deleted from the sequence $v_{\sigma_0}, \ldots, v_{\sigma_k}$.
\begin{defn}\label{v2}
If $K$ and $L$ are simplicial complexes, then a simplicial map $f:K\rightarrow L$ is a map $f:|K|\rightarrow|L|$ which takes vertices of any given simplex of $K$ into vertices of some simplex of $L$ and is affine on each simplex, i.e., $f(\sum_{i}\lambda_iv_{\sigma_i})=\sum_i\lambda_if(v_{\sigma_i}).$
\end{defn}
\begin{defn}
For $x\in |K|$, for a simplicial complex $K$, we define the carrier of $x$ (carr$(x)$) to be the smallest simplex of $K$ containing $x$.
\end{defn}
\begin{defn}\label{v3}
Let $f:|K|\rightarrow|L|$ be continuous. A simplicial approximation to $f$ is a simplicial map $g:K\rightarrow L$ such that $g(x)\in \mbox{carr}(f(x))$ for each $x\in |K|$.
\end{defn}
\begin{cor}\cite[\rom{4} Corollary 22.4]{Bredon}\label{corol11}
If $g$ is a simplicial approximation to $f$, then $f$ is homotopic to $g$.
\end{cor}
\begin{thm}\cite[\rom{4} Theorem 22.10]{Bredon}\label{v1}
  Let $K$ and $L$ be finite simplicial complexes and let $f:|K|\rightarrow|L|$ be continuous. Then there is an integer $t\ge0$ such that there is a simplicial approximation $g:K^{[t]}\rightarrow L$ to $f$.
\end{thm}

 \noindent{\bf Proof of Lemma \ref{y2}.} Assume that $f:|K|\rightarrow |K|$ is homotopic to the identity map. We take a non-zero $r$-cycle $\sigma=\sum_{i=1}^k\sigma_i \in H_r(|K|;\Z_2)=\{\tau \in C_r(K;\Z_2) | \partial_r \tau=0\}$, where $\sigma_i$ is a $r$-simplex of $K$.  If $\cup_{i=1}^k\big(\sigma_i(\Delta_r)\big)\not\subset f(|K|)$, since $K$ is a finite simplicial complex, there is a number $\delta>0$ and $x_0\in\cup_{i=1}^k\big(\sigma_i(\Delta_r)\big)$ such that $\operatorname{dist}(x_0, f(|K|))>\delta.$ Subdivide $K$ so that mesh$(K)<\delta/2$. We may assume that this holds for $K$, and that for such $K$ (mesh$(K)<\delta/2$), there exist $\delta>0$ and $x_0\in\cup_{i=1}^k\big(\sigma_i(\Delta_r)\big)$ such that $\operatorname{dist}(x_0, f(|K|))>\delta$. By Theorem \ref{v1}, there is a simplicial approximation $g:K^{[s]}\rightarrow K$ to $f$, where $s\ge 0$. It follows that for $x\in|K|$, we have $g(x)\in \mbox{carr}(f(x))$, that is, $\operatorname{dist}(g(x),f(x))\le\mbox{mesh}(K)$. We claim that $\cup_{i=1}^k \mbox{Int}\big(\sigma_i(\Delta_r)\big)\subset g(|K|)$. If the claim does not hold true, we may assume that there exists $z\in \mbox{Int}\big(\sigma_1(\Delta_r)\big)$ such that $z\notin g(|K|)$. Since $g$ is a simplicial map, we have
\begin{equation}\label{upr}
\mbox{Int}\big(\sigma_1(\Delta_r)\big)\cap g(\upsilon)=\emptyset \mbox{ for all } \upsilon\in K^{[s]}.
\end{equation}  Since $f$ is homotopic to the identity map, by Corollary \ref{corol11}, $(g)_\ast:H_r(|K|;\Z_2)\rightarrow H_r(|K|;\Z_2)$ is an isomorphism. Then for $\sigma\in H_r(|K|;\Z_2)$, there exists $\tilde{\sigma}\in H_r(|K|;\Z_2)$ such that $g_\ast(\tilde{\sigma})=\sigma$. Since $K$ is simplicial $r$-complex, we have $g_\Delta(\tilde{\sigma})=\sigma$ in $C_r(K;\Z_2)$, which is a contradiction to \eqref{upr}. Indeed, writing $\tilde{\sigma}=\sum_{j=1}^{\ell} \tilde{\sigma}_j=\sum_{j=1}^\ell\langle v_0^j, \cdots, v_r^j\rangle$, we see that $g_\Delta(\tilde{\sigma})=\sum_{j=1}^\ell \langle g(v_0^j),\cdots, g(v_r^j)\rangle$. Then, by \eqref{upr}, $\langle g(v_0^j),\cdots, g(v_r^j)\rangle\neq \sigma_1$ in $C_r(K;\Z_2)$ for all $j=1, \cdots, \ell$. Thus, we deduce that  for  $z_0\in |K|$ satisfying $x_0=g(z_0)$,
$$
\delta<\operatorname{dist}(x_0, f(|K|))\le  \operatorname{dist}(x_0, f(z_0))\le \operatorname{dist}(x_0, g(z_0))+ \operatorname{dist}(f(z_0), g(z_0))<\frac{\delta}{2},
$$
 which is a contradition.
$\Box$
\begin{cor}\label{iu}
Let  $\phi:|K|\rightarrow X$ be a homeomrphism, where $X$ is a space and $|K|$ is  a finite $r$-dimensional polyhedron  with $H_r(|K|;\Z_2)\neq 0$ for some $r\in \N$. Define 
\[
\bar{\mathcal{C}}=\{\bar{f}:X\rightarrow X \mbox{, continuous } |\ \bar{f} \mbox{ is homotopic to the identity map }\}.
\]  
 We  take a non-zero $r$-cycle $\sigma=\sum_{i=1}^k\sigma_i\in H_r(|K|;\Z_2)=\{\tau \in C_r(K; \Z_2) | \partial_r \tau=0\}$, where $\sigma_i$ is a $r$-simplex of $K$.  Then  for all $\bar{f}\in \bar{\mathcal{C}},$
\[
 \cup_{i=1}^k (\phi\circ{\sigma}_i)(\Delta_r)\subset \bar{f}(X),
\]
where $\Delta_r=\{x=\sum_{i=0}^r \lambda_i e_i\ |\ \sum_{i=0}^r\lambda_i=1,0\le \lambda_i\le 1\}$, and $e_0, e_1, \cdots, e_r$ are the standard basis of $\R^{r+1}$.
\end{cor}
\begin{proof}
Let $\phi:|K|\rightarrow X$ be a homeomorphism and $0\neq \sigma=\sum_{i=1}^k\sigma_i\in H_r(|K|;\Z_2)=\{\tau \in C_r(K; \Z_2) | \partial_r \tau=0\}$. If $\bar{f}:X\rightarrow X$ is homotopic to the identity map, then  $f\equiv \phi^{-1}\circ \bar{f}\circ \phi:|K|\rightarrow |K|$ is  homotopic to the identity map.
By Lemma \ref{y2}, we see that $\cup_{i=1}^\ell \big(\sigma_i(\Delta_r)\big) \subset f(|K|)$. Thus, we have $\cup_{i=1}^k \phi\big(\sigma_i(\Delta_r)\big)=\phi\big(\cup_{i=1}^k \sigma_i(\Delta_r)\big) \subset \phi\circ f(|K|)=\bar{f}(X)$. 
\end{proof}
\begin{rmk}
In Lemma \ref{y2}, the condition that $K$ is a finite simplicial $r$-complex  with $H_r(|K|;\Z_2)\neq 0$ for some $r\in \N$ is necessary. Indeed, denoting $A(1,2)=\{x\in \R^2\ | \ 1\le |x|\le 2\}$, we define a map $f:A(1,2)\rightarrow A(1,2)$ by $f(x)=x/|x|$, where $f$ is homotopic to the identity map. We take $\{\sigma_i\}_{i=1}^k$ such that $\sigma_i:\Delta_1\rightarrow A(1,2)$ is a singular $1$-simplex of $A(1,2)$ and  $\cup_{i=1}^k \sigma_i(\Delta_1)=\{x\in \R^2\ |\ |x|=2\}$. Then we see that $0\neq [[\sigma]]=[[\sum_{i=1}^k \sigma_i]]\in H_1(A(1,2);\Z_2)$, but $f(A(1,2))\cap \{x\in \R^2\ |\ |x|=2\}$ is empty.
\end{rmk}

Following \cite{Bredon}, we denote $M$ a topological $n-$manifold. Let $A\subset M$ be a closed set and let $x\in A$. Let $G$ be any coefficient group and denote by
\[
j_{x,A}:H_n(M,M\setminus A; G)\rightarrow H_n(M, M\setminus\{x\}; G),
\]
the map induced from the inclusion.
\begin{prop}\cite[\rom{6} Proposition 7.1]{Bredon}\label{w2}
  If $A$ is a compact, convex subset of $\R^n\subset M$ then $j_{x,A}$ is an isomorphism and both groups are isomorphic to $G$.
\end{prop}
\begin{defn}\label{w1}
Let $\Theta_x\otimes G=H_n(M, M\setminus \{x\}; G)\approx G.$ Also let $\Theta\otimes G=\cup\{ \Theta_x\otimes G|x\in M\}$ (disjoint) and let $p:\Theta\otimes G\rightarrow M$ be the function taking $\Theta_x\otimes G$ to $x$. Give $\Theta\otimes G$ the following topology: Let $U\subset M$ be open and $\alpha\in H_n(M, M\setminus\overline{U};G).$ Then for $x\in U, j_{x,\overline{U}}(\alpha)\in \Theta_x\otimes G$. Let $U_\alpha=\{ j_{x,\overline{U}}(\alpha)|x\in U\}\subset p^{-1}(U)\subset \Theta\otimes G$. Take the $U_\alpha$ as a basis for the topology on $\Theta \otimes G$.
\end{defn}
\begin{prop}\cite[\rom{6} Proposition 7.3]{Bredon}
  The sets $U_\alpha$ defined in Definition \ref{w1} are the basis of a topology.
\end{prop}
\begin{defn}
For $A\subset M$ closed, the group of sections over $A$ of $\Theta \otimes G$ is
\[
\Gamma(A,\Theta\otimes G)=\{ s:A\rightarrow\Theta\otimes G, \mbox{ continuous }\ |\ p \circ s=1\}.
\]
This is an abelian group under the operation $(s+s^\prime)(x)=s(x)+s^\prime(x).$ Also, let $\Gamma_c(A,\Theta\otimes G)$ be the subgroup consisting of sections with compact support, i.e., those sections with value $0$ outside some compact set.
\end{defn}
\begin{thm}\cite[\rom{6} Theorem 7.8]{Bredon}\label{w3}
  Let $M$ be a topological $n$-manifold and let $A\subset M$ be closed. Then the map $J_A: H_n(M, M\setminus A; G)\rightarrow \Gamma_c(A,\Theta\otimes G)$ given by $J_A(\alpha)(x)=j_{x,A}(\alpha)$ is an isomorphism.
\end{thm}
\begin{cor}\cite[\rom{6} Corollary 7.12]{Bredon}
If $M$ is a connected compact topological  $n$-manifold without boundary, then $H_n(M;\Z_2)=\Z_2$
\end{cor}
\noindent{\bf Proof of Lemma \ref{y1}.}  Suppose that there exists $x\in M\setminus f(M)$. Then we may assume that there exists a compact, convex set $U$ of $\R^n\subset M$ such that $U\cap f(M)=\emptyset$. For the inclusions $i:M\setminus U\hookrightarrow M$ and $j:M \hookrightarrow(M,M\setminus U)$, we have the commutative diagram
  \begin{equation}\label{u1}
  \xymatrix{ & H_n(M; \Z_2)\ar[d]_{\tilde{f}_\ast}\ar[dr]^{f_\ast}\\
   \cdots  \ar[r]^{\partial_\ast \, \qquad} & H_n(M\setminus U; \Z_2)\ar[r]^{\quad i_\ast}&  H_n(M; \Z_2) \ar[r]^{j_\ast\quad \,} & H_n(M, M\setminus U; \Z_2) \ar[r]^{\quad \qquad \partial_\ast}& \cdots.}
  \end{equation}
   Let $e_x:\Gamma(M, \Theta\otimes \Z_2)\rightarrow H_n(M,M\setminus \{x\}; \Z_2)$ be the evaluation map given by $e_x(s)=s(x)$. We note that, since $M$ is connected, each section is uniquely determined by its value at one point. Then, by Theorem \ref{w3}, we see that $e_x$ is an isomorphism. Therefore $j_{x,M}=e_x\circ J_M$ is an isomorphism. By Proposition \ref{w2}, we have the map $j_{x,U}: H_n(M, M\setminus U; \Z_2)\rightarrow H_n(M, M\setminus \{x\}; \Z_2)$ is an isomorphism. Therefore, from the following commutative diagram
  \[
  \xymatrix{H_n(M; \Z_2)\ar[rd]_{j_{x,M}} \ar[r]^{j_\ast\quad} &H_n(M, M\setminus U; \Z_2)\ar[d]^{j_{x,U}}\\
                     &H_n(M, M\setminus\{x\}; \Z_2),}
  \]
  we deduce that $j_\ast$ is an isomorphism. By the exactness of the sequence
  \[
  \cdots \xrightarrow{\partial_\ast} H_n(M\setminus U; \Z_2) \xrightarrow{i_\ast}  H_n(M; \Z_2) \xrightarrow{j_\ast} H_n(M, M\setminus U; \Z_2) \xrightarrow{\partial_\ast} \cdots,
  \]
  we deduce that $i_\ast=0$. Since $f$ is homotopic to the identity map, $f_\ast$ is an isomorphism. By the commutative diagram \eqref{u1}, we see that $f_\ast=i_\ast\circ \tilde{f}_\ast$. Then we obtain $H_n(M; \Z_2)=0$, which is a contradiction.
$\Box$


\begin{thebibliography}{00}
\bibitem{ABC} A. Ambrosetti, M. Badiale, S. Cingolani, Semiclassical states of nonlinear Schrödinger equations, Arch. Rational Mech. Anal.  140  (1997),  no. 3, 285-300.
\bibitem{AMN}
    A. Ambrosetti, A. Malchiodi, W.-M. Ni, Singularly perturbed elliptic equations with symmetry: existence of solutions concentrating on spheres, I, Comm. Math. Phys.  235  (2003),  no. 3, 427-466.
\bibitem{Bredon}
 G. E. Bredon, Topology and geometry, Graduate Texts in Mathematics, 139. Springer-Verlag, New York, 1993.
\bibitem{BL} H. Berestycki, P.-L. Lions, Nonlinear scalar field equations I. Existence of a ground state, Arch. Rational Mech. Anal. 82 (1983), no. 4, 313 - 345.
\bibitem{B}
 J. Byeon, Standing waves for nonlinear Schrödinger equations with a radial potential, Nonlinear Anal.  50  (2002),  no. 8, Ser. A: Theory Methods, 1135-1151.
\bibitem{BJ1} J. Byeon, L. Jeanjean, Standing waves for nonlinear Schrödinger equations with a general nonlinearity, Arch. Ration. Mech. Anal.  185  (2007),  no. 2, 185-200.
\bibitem{BJ2} J. Byeon, L. Jeanjean, Multi-peak standing waves for nonlinear Schrödinger equations with a general nonlinearity, Discrete Contin. Dyn. Syst.  19  (2007),  no. 2, 255-269.
\bibitem{BJT}
J. Byeon, L. Jeanjean, K. Tanaka, Standing waves for nonlinear Schrödinger equations with a general nonlinearity: one and two dimensional cases, Comm. Partial Differential Equations  33  (2008),  no. 4-6, 1113–1136.
\bibitem{BT1} J. Byeon, K. Tanaka, Semi-classical standing waves for nonlinear Schrödinger equations at structurally stable critical points of the potential, J. Eur. Math. Soc. (JEMS)  15  (2013),  no. 5, 1859-–1899.
\bibitem{BT2} J. Byeon, K. Tanaka, Semiclassical standing waves with clustering peaks for nonlinear Schrödinger
equations, Mem. Amer. Math. Soc. 229(1076) (2014)
\bibitem{BT3}
J. Byeon, K. Tanaka, Multi-bump positive solutions for a nonlinear elliptic problem in expanding tubular domains, Calc. Var. Partial Differential Equations  50  (2014),  no. 1-2, 365-397.
\bibitem{BW} J. Byeon, Z.-Q. Wang, Standing waves with a critical frequency for nonlinear Schrödinger equations, Arch. Ration. Mech. Anal.  165  (2002),  no. 4, 295-316.
\bibitem{BW2}
    J. Byeon, Z.-Q. Wang, Standing waves with a critical frequency for nonlinear Schrödinger equations, II, Calc. Var. Partial Differential Equations  18  (2003),  no. 2, 207-219.
\bibitem{CR} V. Coti Zelati, P. H. Rabinowitz, Homoclinic type solutions for a semilinear elliptic PDE on $\R^n$, Comm. Pure Appl. Math.  45  (1992),  no. 10, 1217-1269.
\bibitem{CJT}
S. Cingolani, L. Jeanjean, K. Tanaka, Multiplicity of positive solutions of nonlinear Schrödinger equations concentrating at a potential well, Calc. Var. Partial Differential Equations 53 (2015), no. 1-2, 413-439.
\bibitem{DLY} E. N. Dancer, K. Y. Lam,  S. Yan,  The effect of the graph topology on the existence of multipeak solutions for nonlinear Schrödinger equations, Abstr. Appl. Anal.  3  (1998),  no. 3--4, 293–318.
\bibitem{DPR}
P. d'Avenia, A. Pomponio, D. Ruiz, Semiclassical states for the nonlinear Schrödinger equation on saddle points of the potential via variational methods, J. Funct. Anal.  262  (2012),  no. 10, 4600-4633.

\bibitem{DF1}
M. del Pino, P. Felmer, Local mountain passes for semilinear elliptic problems in unbounded domains, Calc. Var. Partial Differential Equations  4  (1996),  no. 2, 121-–137.
\bibitem{DF2}
M. del Pino, P. Felmer, Semi-classical states for nonlinear Schrödinger equations, J. Funct. Anal.  149  (1997),  no. 1, 245-–265.
\bibitem{DF3}
M. del Pino, P. Felmer, Multi-peak bound states for nonlinear Schrödinger equations, Ann. Inst. H. Poincaré Anal. Non Linéaire  15  (1998),  no. 2, 127-149.
\bibitem{DF4}
M. del Pino, P. Felmer, Semi-classical states of nonlinear Schrödinger equations: a variational reduction method, Math. Ann.  324  (2002),  no. 1, 1-32.
\bibitem{DKW}
M. del Pino, M. Kowalczyk, J. Wei, Concentration on curves for nonlinear Schrödinger equations, Comm. Pure Appl. Math.  60  (2007),  no. 1, 113-146.


\bibitem{FW}     A. Floer, A. Weinstein,  Nonspreading wave packets for the cubic Schrödinger equation with a bounded potential, J. Funct. Anal.  69  (1986),  no. 3, 397-408.
\bibitem{GNN} B. Gidas, W. M. Ni, L. Nirenberg, Symmetry and related properties via the maximum principle, Comm. Math. Phys.  68  (1979), no. 3, 209-243.
\bibitem{H} M. W. Hirsch, Differential topology,  Graduate Texts in Mathematics, 33. Springer-Verlag, New York, 1994.  
\bibitem{JT} L. Jeanjean, K. Tanaka, A remark on least energy solutions in $\R^N$, Proc. Amer. Math. Soc.  131,  (2003),  no. 8, 2399-2408.
\bibitem{JT2}
    L. Jeanjean, K. Tanaka, Singularly perturbed elliptic problems with superlinear or asymptotically linear nonlinearities, Calc. Var. Partial 
Differential Equations  21  (2004),  no. 3, 287-318.
\bibitem{K}
M. A. Kervaire, A manifold which does not admit any differentiable structure, Comment. Math. Helv.  34  (1960) 257-270.
\bibitem{roche}
Z. Kopal, Close Binary Systems, Chapman and Hall, London (1959)
\bibitem{KW}
X. Kang, J. Wei, On interacting bumps of semi-classical states of nonlinear Schrödinger equations, Adv. Differential Equations  5  (2000),  no. 7--9, 899-928.
\bibitem{JML}
J. M. Lee,  Introduction to smooth manifolds, Second edition, Graduate Texts in Mathematics, 218. Springer, New York, 2013. 
\bibitem{L}
Y. Y. Li, On a singularly perturbed elliptic equation, Adv. Differential Equations  2  (1997),  no. 6, 955-980.
\bibitem{LS}
Y. Lee, J. Seok, Multiple interior and boundary peak solutions to singularly perturbed nonlinear Neumann problems under the Berestycki-Lions condition, Math. Ann.  367  (2017),  no. 1-2, 881-928.
\bibitem{M}
C. Manolescu, Pin(2)-equivariant Seiberg-Witten Floer homology and the triangulation conjecture, J. Amer. Math. Soc.  29  (2016),  no. 1, 147-176.
\bibitem{M1}
P. Meystre, Atom Optics, Springer, 2001.
\bibitem{M2}
D. L. Mills, Nonlinear Optics, Springer, 1998.
\bibitem{O1}
Y.-G. Oh, Existence of semiclassical bound states of nonlinear Schrödinger equations with potentials of the class $(V)_a$, Comm. Partial Differential Equations  13  (1988),  no. 12, 1499-1519.
\bibitem{O2}
Y.-G. Oh, On positive multi-lump bound states of nonlinear Schrödinger equations under multiple well potential, Comm. Math. Phys.  131  (1990),  no. 2, 223–253.
\bibitem{P}
R. S. Palais, The principle of symmetric criticality, Comm. Math. Phys.  69  (1979), no. 1, 19-30.
\bibitem{PS}
L. P. Pitaevskii,   S. Stringari,  Bose-Einstein condensation, International Series of Monographs on Physics, 116. The Clarendon Press, Oxford University Press, Oxford, 2003.
\bibitem{R}
P. H. Rabinowitz, On a class of nonlinear Schrödinger equations, Z. Angew. Math. Phys.  43  (1992),  no. 2, 270-291.
\bibitem{W}
X. Wang, On concentration of positive bound states of nonlinear Schrödinger equations, Comm. Math. Phys.  153  (1993),  no. 2, 229-244.
\bibitem{W1}
Z. Q. Wang, Existence and symmetry of multi-bump solutions for nonlinear Schrödinger equations, J. Differential Equations  159  (1999),  no. 1, 102-137.

\end{thebibliography}
\end{document}